\documentclass[a4paper,12pt,reqno]{amsart}

\usepackage{amsmath,amssymb,amsthm}
\usepackage{latexsym}
\usepackage{color}
\usepackage{graphicx}
\usepackage{mathrsfs}
\usepackage{enumerate}
\usepackage[abbrev]{amsrefs}
\usepackage[T1]{fontenc}
\usepackage{mathtools}
\mathtoolsset{showonlyrefs=true}
\usepackage[colorlinks,
linkcolor=blue,
anchorcolor=blue,
citecolor=blue]{hyperref}

\allowdisplaybreaks[4]

\theoremstyle{plain}
\newtheorem{thm}{Theorem}[section]
\newtheorem{lemm}[thm]{Lemma}
\newtheorem{prop}[thm]{Proposition}

\theoremstyle{definition}
\newtheorem{df}[thm]{Definition}
\newtheorem{rem}[thm]{Remark}

\makeatletter

\@addtoreset{equation}{section}
\makeatother

\renewcommand{\div}{\operatorname{div}}
\newcommand{\dB}{\dot{B}}

\newcommand{\hf}{\widehat{f}}

\renewcommand{\leq}{\leqslant}
\newcommand{\tC}{\widetilde{C}}
\newcommand\defref[1]{Definition~\ref{#1}}
\newcommand\thmref[1]{Theorem~\ref{#1}}
\newcommand\propref[1]{Proposition~\ref{#1}}
\newcommand\lemref[1]{Lemma~\ref{#1}}

\newcommand{\n}[1]{{\left\|#1\right\|}}

\newcommand{\Mp}[1]{\left\{#1\right\}}
\renewcommand{\sp}[1]{\left(#1\right)}
\newcommand{\fB}{\widehat{\dot{B}}{}}
\renewcommand{\leq}{\leqslant}
\renewcommand{\geq}{\geqslant}

\begin{document}
\title[Compressible rotating Navier--Stokes--Korteweg system]
{Global strong solutions to the $3$D rotating compressible Navier--Stokes--Korteweg system for large data in the critical $\widehat{L^p}$ framework}
	\author[M.~Fujii]{Mikihiro Fujii}
	\address[M.~Fujii]{Graduate School of Science, Nagoya City University, Nagoya, 467-8501, Japan}
	\email[M.~Fujii]{fujii.mikihiro@nsc.nagoya-cu.ac.jp}
	\author[S.~Zhang]{Shunhang Zhang}
	\address[S.~Zhang]{School of Mathematics and Computational Science, Wuyi University, Jiangmen, 529020, Guangdong, China}
	\email[S.~Zhang]{zhangshunhang@wyu.edu.cn}
\begin{abstract}
Let us consider the $3$D compressible Navier--Stokes--Korteweg system in the rotational framework. 
Although there is a wealth of literature on the weak solutions to this system,
there seem to be no results on the strong solutions.
In this paper, we show the unique existence of global solutions for {\it large} initial data in the critical Besov-type spaces based on the Fourier--Lebesgue spaces $\widehat{L^p}(\mathbb{R}^3)$ with $2 \leq p < 3$, provided that the rotation speed and the Mach number are sufficiently large and small, respectively.
The key ingredient of the proof is to establish the Strichartz-type estimates due to the dispersion caused by the mixture of the rotation and acoustic waves in the Fourier--Lebesgue spaces, and focus on the better structure of dissipation from the Korteweg term and the nonlinear terms of the divergence form in the momentum formulation.
\end{abstract}
\maketitle

\section{Introduction}
This paper is concerned with the initial value problem of the compressible Navier--Stokes--Korteweg system with the Coriolis force in the whole space $\mathbb{R}^3$, which is one of the models of geophysical flows used to describe the motion of a compressible fluid with capillary effects under the action of the Coriolis force. Regarding the further physical background, we refer to \cites{MR3292658,MR775366,MR2186374} and references therein.
The governing system is given as follows:
\begin{equation}\label{eq:CNSKC-0}
	\begin{cases}
		\partial_t\rho+\operatorname{div}(\rho u)=0,& t>0, x\in \mathbb{R}^3,\\
		\begin{aligned}
			\partial_t (\rho u) & + \div (\rho u\otimes u) + \Omega ( e_3 \times (\rho u) ) + \dfrac{1}{\varepsilon^2}\nabla P(\rho)\\
			&
			=
			\mu \Delta u + (\mu+\lambda) \nabla \div u
			+ \div \mathcal{K}(\rho),
		\end{aligned}
		& t>0, x\in \mathbb{R}^3,\\
		\rho(0,x)=\rho_0(x),\quad u(0,x)=u_0(x), & x\in \mathbb{R}^3.
	\end{cases}
\end{equation}
Here, since \eqref{eq:CNSKC-0} is physically well{-}justified at mid-latitude regions, 
the centrifugal force is canceled by the natural assumption that the centrifugal force balances with the geostrophic force.
The unknowns $\rho=\rho(t,x)\in (0,\infty)$ and $u=u(t,x)\in \mathbb{R}^3$ with $(t,x)\in [0,\infty)\times \mathbb{R}^3$ stand for the fluid density and the velocity field, respectively, while  $\rho_0=\rho_0(x)\in (0,\infty)$ and $u_0=u_0(x)\in \mathbb{R}^3$ with $x\in\mathbb{R}^3$ are the supplemented initial data. 
The constants $\mu$ and $\lambda$ are shear and bulk viscosity coefficients, respectively, and satisfy $\mu>0$ and $\nu:=2\mu+\lambda>0$. The positive constant $\varepsilon$ is called the Mach number, which represents the reciprocal of the sound speed. 
One of the key terms in \eqref{eq:CNSKC-0} is the Coriolis force $\Omega(e_3\times(\rho u))$, which represents the rotation of the fluid around the vertical axis $e_3=(0,0,1)^\top$ with the constant angular speed $\Omega \in \mathbb{R}$. 
The Korteweg tensor $\mathcal{K}(\rho)$ is defined by
\begin{equation}
	\begin{split}
		\mathcal{K}(\rho):=\frac{\kappa}{2}(\Delta\rho^2-|\nabla \rho|^2)\operatorname{Id}-\kappa\nabla\rho\otimes\nabla\rho,
	\end{split}
\end{equation}where the constant $\kappa>0$ is the capillary coefficient. Hence, one may observe that the capillary term $\div \mathcal{K}(\rho)$ is expressed by\begin{equation}
\begin{split}
	\div \mathcal{K}(\rho)=\kappa\rho\nabla\Delta \rho.
\end{split}
\end{equation}
The pressure $P=P(\rho):(0,\infty) \to \mathbb{R}$ is a given function real analytic at $\rho=\rho_{\infty}$ and satisfying
\begin{align}\label{monotone-P}
    P^\prime(\rho_{\infty})>0
\end{align}
with some positive constant $\rho_{\infty}>0$. 
This paper treats the system \eqref{eq:CNSKC-0} as the momentum formulation.
More precisely, we define the momentum $m:=\rho u$ and then $(\rho,m)$ should solve the following system:
\begin{equation}\label{eq:CNSKC-1}
	\begin{cases}
		\partial_t\rho+\operatorname{div}m=0,& t>0, x\in \mathbb{R}^3,\\
		\begin{aligned}
			\partial_t m & + \div \sp{ \dfrac{m \otimes m}{\rho} } + \Omega ( e_3 \times m ) + \dfrac{1}{\varepsilon^2}\nabla P(\rho)\\
			&
			=
			\mu \Delta \sp{\dfrac{m}{\rho}} + (\mu+\lambda) \nabla \div \sp{\dfrac{m}{\rho}}
			+ \div\mathcal{K}(\rho),
		\end{aligned}
		& t>0, x\in \mathbb{R}^3,\\
		\rho(0,x)=\rho_0(x),\quad m(0,x)=m_0(x), & x\in \mathbb{R}^3.
	\end{cases}
\end{equation}
The clear difference from the velocity formulation is that now the nonlinear terms are of the divergence form, which plays a key role in our calculus. 
It is well-known that the system \eqref{eq:CNSKC-1} has an invariant scaling structure; 
if $(\rho,m)$ is a solution to \eqref{eq:CNSKC-1}, then the scaled functions
\begin{align}\label{scaling} 
\rho_{\lambda}(t,x):= \rho (\lambda^2 t,\lambda x),
\quad 
m_{\lambda}(t,x):=\lambda m (\lambda^2 t,\lambda x)
\end{align}
also solve \eqref{eq:CNSKC-1} with the pressure $P$ replaced by $\lambda^2 P$ for all $\lambda>0$. 
A norm space $X$ of distributions on $\mathbb{R}^3$ is called a critical space with respect to the scaling \eqref{scaling} if
$\n{\sp{\rho_{\lambda},m_{\lambda}}(0,\cdot)}_{X}
=
\n{\sp{\rho,m}(0,\cdot)}_{X}$
for all $\lambda>0$.
For instance, $\dB_{p,q}^{\frac{3}{p}}(\mathbb{R}^3) \times \dB_{p,q}^{\frac{3}{p}-1}(\mathbb{R}^3)$ and $\fB_{p,q}^{\frac{3}{p}}(\mathbb{R}^3) \times \fB_{p,q}^{\frac{3}{p}-1}(\mathbb{R}^3)$ with $1 \leq p,q \leq  \infty$
are the corresponding critical Besov spaces and critical Fourier--Besov spaces\footnote{See Section \ref{sec:pre} for the precise definition.}, respectively.

The aim of this paper is to consider the solutions to \eqref{eq:CNSKC-1} around the constant equilibrium state $(\rho_{\infty},0)$, where $\rho_{\infty}$ is the positive constant appearing in \eqref{monotone-P}, and prove the global well-posedness of \eqref{eq:CNSKC-1} for arbitrarily large initial perturbation in the critical Fourier--Besov spaces framework $\fB_{p,1}^{\frac{3}{p}}(\mathbb{R}^3) \times \fB_{p,1}^{\frac{3}{p}-1}(\mathbb{R}^3)$ with $2 \leq p < 3$ and some low-frequency assumptions, provided that the $|\Omega|$ and $\varepsilon$ are sufficiently large and small, respectively.

As we consider the solutions around the constant equilibrium state $(\rho_{\infty},0)$,
we further reformulate the system \eqref{eq:CNSKC-1}.
Without loss of generality, by performing a suitable scaling transform, we reduce the study on $\eqref{eq:CNSKC-1}$ to the case $\rho_\infty=P^\prime(\rho_\infty)=\nu=1$; see \cite{MR4865748} for the details on this scaling transformation. 
Let us define the perturbation of the density around the constant $\rho_{\infty}=1$ by
\begin{equation}
	\begin{split}
		a(t,x):=\frac{\rho(t,x)-1}{\varepsilon},
        \qquad
        a_0(x)=\frac{\rho_0(x)-1}{\varepsilon}
	\end{split}
\end{equation}
and assume that the given datum $a_0$ is independent of $\varepsilon$.
Then, we reformulate \eqref{eq:CNSKC-1} as 
\begin{equation}\label{eq:CNSKC-2}
	\begin{cases}
		\partial_t a+\dfrac{1}{\varepsilon} \div m=0,
		& t>0, x\in \mathbb{R}^3,\\\begin{aligned}
			\partial_t m-\mu\Delta m-(\mu&+\lambda)\nabla\div m+\Omega ( e_3 \times m )\\&+\dfrac{1}{\varepsilon} \nabla a-\kappa\varepsilon\nabla\Delta a=N_\varepsilon[a,m],
		\end{aligned}
		& t>0, x\in \mathbb{R}^3,\\
		a(0,x)=a_0(x),\quad m(0,x)=m_0(x), & x\in \mathbb{R}^3,
	\end{cases}
\end{equation}
where
\begin{equation}
	\begin{split}
		N_\varepsilon[a,m]:=&\div\left((I(\varepsilon a)-1)m\otimes m\right)-\mu\Delta(I(\varepsilon a)m)\\
		&-(\mu+\lambda)\nabla\div(I(\varepsilon a)m)-\frac{1}{\varepsilon}{J}(\varepsilon a)\nabla a+\kappa\varepsilon^2\nabla(a\Delta a)\\&+\frac{\kappa\varepsilon^2}{2}\nabla(|\nabla a|^2)-\kappa\varepsilon^2\div(\nabla a\otimes\nabla a)
	\end{split}
\end{equation}with\begin{equation}
	\begin{split}
		I(b):=\frac{b}{1+b},\quad J(b):=P^\prime(1+b)-1.
	\end{split}
\end{equation}
\par 
Before presenting our main theorem, we first give a brief overview of existing results relevant to the present study.  
The pioneering contributions of Danchin \cites{MR1779621,MR1855277} are that he established the well-posedness of the compressible Navier--Stokes system (i.e. $\Omega=\kappa=0$ in \eqref{eq:CNSKC-0}) in the critical $L^2$ framework by exploiting Sobolev embeddings in the endpoint case for Besov spaces.  
Building upon this, Danchin and Desjardins \cite{MR1810272} investigated the compressible Navier--Stokes--Korteweg system (i.e. $\Omega=0$ in \eqref{eq:CNSKC-0}): 
\begin{align}\label{eq:CNSK}
	\begin{cases}
		\partial_t \rho + \div (\rho u) = 0, & t>0,x \in \mathbb{R}^3,\\
		\begin{aligned}
\partial_t (\rho u) + \div (\rho u &\otimes u)+\frac{1}{\varepsilon^2}\nabla P(\rho) \\
&= \mu \Delta u + (\mu+\lambda) \nabla \div u+ \kappa \rho \nabla \Delta \rho, 
		\end{aligned}
		& t>0,x \in \mathbb{R}^3,  \\
		\rho(0,x)=\rho_0(x),\quad
		u(0,x)=u_0(x), & x \in \mathbb{R}^3,
	\end{cases}
\end{align}
and proved that if the initial perturbation $(a_0,u_0)=(\rho_0-1,u_0)$ is sufficiently small in the critical Besov space  $(\dB_{2,1}^{\frac{1}{2}}(\mathbb{R}^3) \cap \dB_{2,1}^{\frac{3}{2}}(\mathbb{R}^3)) \times \dB_{2,1}^{\frac{1}{2}}(\mathbb{R}^3)$,
then \eqref{eq:CNSK} has a unique global solution $(a,u)=(\rho-1,u)$ satisfying  
\[
a,\nabla a, u \in C([0,\infty);\dB_{2,1}^{\frac{1}{2}}(\mathbb{R}^3)) \cap L^1(0,\infty;\dB_{2,1}^{\frac{5}{2}}(\mathbb{R}^3)).
\]
Subsequently, the extended results within the high-frequency $L^p$ setting were established in \cites{MR4312286,MR4340485}, which are comparable to those for the compressible Navier-Stokes system in \cites{MR2679372,MR2823668}.  
It is worth noting that the $L^1$-maximal regularity  
\[
a,\nabla a \in L^1(0,\infty;\dB_{2,1}^{\frac{5}{2}}(\mathbb{R}^2))
\]
for the density perturbation stems from the dissipation generated by the Korteweg term $\kappa \nabla \Delta a$.  
This extra dissipative structure yields stronger regularity properties compared to those for the standard compressible Navier--Stokes equations.  
In fact, \cite{MR4340485} demonstrated that the small global solutions to \eqref{eq:CNSK} enjoy the analytic Gevrey regularity.  
Furthermore, Kobayashi and Nakasato \cite{kobayashi2025time}, working in the critical Fourier--Besov framework, improved the $L^2$-regularity for the low-frequency part in \cite{MR4312286} and proved global well-posedness of the momentum formulation of \eqref{eq:CNSK} for small initial perturbation $(a_0,m_0)$ in  $(\fB_{p,1}^{\frac{3}{p}-1}(\mathbb{R}^3) \cap \fB_{p,1}^{\frac{3}{p}}(\mathbb{R}^3)) \times \fB_{p,1}^{\frac{3}{p}-1}(\mathbb{R}^3)$ with $1 \leq p \leq \infty$.
On the other hand, it is also a significant topic to consider the singular parameter limits for the compressible flows. Concerning the low Mach number limit of the compressible Navier--Stokes system, Danchin \cites{MR1886005,MR3563240} proved that the small global solutions in the critical Besov space converge to the incompressible Navier--Stokes flow as the Mach number $\varepsilon\to 0$. In \cite{MR4721788}, the first author weakened the smallness condition on the initial data in the $2$D case, showing that for arbitrarily large initial perturbation $(a_0,u_0)$ in the critical Besov space $
(\dB_{2,1}^0(\mathbb{R}^2) \cap \dB_{2,1}^1(\mathbb{R}^2)) \times \dB_{2,1}^0(\mathbb{R}^2)$,
the compressible Navier--Stokes system has a unique global solution provided that the Mach number $\varepsilon$ is sufficiently small; moreover, the solution converges to the corresponding large incompressible flow in certain critical space-time norms as $\varepsilon\to 0$.  
This result was recently extended to the $2$D compressible Navier--Stokes--Korteweg system by the first author and Li \cite{MR4836084}, in which they established the unique global solvability in the framework of critical Fourier--Besov spaces  
$
(\fB_{p,1}^{\frac{2}{p}-1}(\mathbb{R}^2) \cap \fB_{p,1}^{\frac{2}{p}}(\mathbb{R}^2))\times \fB_{p,1}^{\frac{2}{p}-1}(\mathbb{R}^2)
$
for $2 \leq p < 4$ and analyzed the incompressible limit in the low-Mach regime.  
For the incompressible limits of $\lambda\to\infty$ type, see \cites{MR3907737,MR3709125,MR3935027,MR4484351}.
\par 
We next turn to results concerning the incompressible Navier--Stokes--Coriolis system (i.e. $\rho\equiv 1$ and $\kappa=0$ in \eqref{eq:CNSKC-0}):
\begin{align}\label{eq:incomp-coriolis}
	\begin{dcases}
		\partial_t u - \mu \Delta u + \Omega( e_3 \times u ) + (u \cdot \nabla)u + \nabla p = 0, & \qquad t > 0, x \in \mathbb{R}^3,\\
		\div u = 0, & \qquad t \geqslant 0, x \in \mathbb{R}^3,\\
		u(0,x) = u_0(x), & \qquad x \in \mathbb{R}^3,
	\end{dcases}
\end{align}
which was first investigated by Babin, Mahalov and Nicolaenko \cites{MR1480996,MR1752139,MR1855663}.  
Subsequently, Iwabuchi and Takada \cites{MR3096523,MR3468733} and Koh, Lee and Takada \cite{MR3229600} analyzed the dispersive properties of the evolution group $\{e^{\pm i\Omega tD_3/|D|}\}_{t\in\mathbb{R}}$ induced by the Coriolis force.  
In particular, they derived Strichartz estimates for the linearized problem, which implies that certain space-time norms of the linear solution can be made arbitrarily small by choosing the rotation speed $|\Omega|$ sufficiently large.  
By exploiting this observation, Iwabuchi and Takada \cite{MR3096523} and Koh, Lee and Takada \cite{MR3229600} proved the global well-posedness of \eqref{eq:incomp-coriolis} for any divergence-free initial data $u_0\in\dot{H}^s(\mathbb{R}^3)$ with $1/2 \leq s < 9/10$, provided $|\Omega|$ is large enough.  
In contrast to incompressible or non-rotating situations, there are comparatively few results addressing the compressible Navier--Stokes--Coriolis system (i.e. $\kappa=0$ in \eqref{eq:CNSKC-0}):
\begin{equation}\label{eq:CNSKC-5}
	\begin{cases}
		\partial_t\rho+\operatorname{div}(\rho u)=0,& t>0, x\in \mathbb{R}^3,\\
		\begin{aligned}
			\partial_t (\rho u) & + \div (\rho u\otimes u) + \Omega ( e_3 \times (\rho u) ) + \dfrac{1}{\varepsilon^2}\nabla P(\rho)\\
			&
			=
			\mu \Delta u + (\mu+\lambda) \nabla \div u,
		\end{aligned}
		& t>0, x\in \mathbb{R}^3,\\
		\rho(0,x)=\rho_0(x),\quad u(0,x)=u_0(x), & x\in \mathbb{R}^3.
	\end{cases}
\end{equation}  
In the context of weak solutions to \eqref{eq:CNSKC-5} in an infinite slab $\mathcal{D}=\mathbb{R}^2\times(0,1)$, see \cites{MR4412078,MR4322370,MR2888285,MR2964771,MR3200090} for results on existence and the fast rotation limit.  
The first analysis of the strong-solution well-posedness for \eqref{eq:CNSKC-5} in $\mathbb{R}^3$ appears in the first author and Watanabe's work \cite{MR4865748}, where Strichartz estimates for the linearized system were established through dispersion analysis arising from the interplay of the Coriolis effect and acoustic waves.  
These estimates enabled the proof of {\it long-time} existence for \eqref{eq:CNSKC-5} in the critical Besov setting; more precisely, given any initial perturbation
$
(a_0,u_0) \in
(\dot{B}_{2,1}^{\frac{1}{2}}(\mathbb{R}^3) \cap \dot{B}_{2,1}^{\frac{3}{2}}(\mathbb{R}^3)) \times \dot{B}_{2,1}^{\frac{1}{2}}(\mathbb{R}^3)
$
and any $T>0$, there exists a constant $\Omega_T=\Omega_T(a_0,u_0)$ such that \eqref{eq:CNSKC-5} has a unique solution on $[0,T)$ whenever $\Omega_T \leq |\Omega| \leq 1/\varepsilon$.  
Recently, in a new paper \cite{fujii2024global}, they improved the above result to the global well-posedness; more precisely, for any 
$
    (a_0,u_0) \in (\dB_{2,\infty}^{-\frac{3}{2}}(\mathbb{R}^3) \cap \dB_{2,1}^{\frac{3}{2}}(\mathbb{R}^3)) \times (\dB_{2,\infty}^{-\frac{3}{2}}(\mathbb{R}^3) \cap \dB_{2,1}^{\frac{1}{2}}(\mathbb{R}^3))$,
there exist positive constants $\Omega_0=\Omega_0(\mu,P,a_0,u_0)$ and $c_0=c_0(\mu,P,a_0,u_0)$ such that if $\Omega \in \mathbb{R} \setminus \{ 0\}$ and $\varepsilon > 0$ satisfy $\Omega_0 \leq |\Omega| \leq c_0 /\varepsilon$,
then
there exists a unique global solution to \eqref{eq:CNSKC-5}.\par 
The aim of this paper is to improve the result of \cite{fujii2024global} to the critical $\widehat{L^p}$ framework inspired by the first author's previous studies \cites{MR4755730,MR4836084}.
Precisely, we will prove the global well-posedness of the perturbed system \eqref{eq:CNSKC-2} for large initial data $(a_0,m_0)$ in the critical Fourier--Besov space $\fB_{p,1}^{\frac{3}{p}}(\mathbb{R}^3) \times \fB_{p,1}^{\frac{3}{p}-1}(\mathbb{R}^3)$ with an additional low-frequency assumption $(a_0,m_0) \in \fB_{p,1}^{\frac{3}{p}-3}(\mathbb{R}^3)$ for $2 \leq p < 3$, provided that $|\Omega|$ and $\varepsilon$ are sufficiently large and small, respectively, in the sense $1 \ll |\Omega| \ll 1/\varepsilon$.
Here, the Fourier--Besov space $\fB_{p,q}^s(\mathbb{R}^3)$ is the set of all distributions whose Besov norms, with $L^p$-norm replaced by the Fourier--Lebesgue $\widehat{L^p}$-norm, are finite; see the next section for the detailed definition.\par 
Our main result now reads as follows.
\begin{thm}\label{thm:large}
        Let $2 \leq p < 3$ 
   and let
        $a_0 \in \fB_{p,1}^{\frac{3}{p}-3} ( \mathbb{R}^3 ) \cap \fB_{p,1}^{\frac{3}{p}} ( \mathbb{R}^3 )$
        and
        $m_0 \in  \fB_{p,1}^{\frac{3}{p}-3} ( \mathbb{R}^3 ) \cap \fB_{p,1}^{\frac{3}{p}-1} ( \mathbb{R}^3 ) $.
        There exists a positive constant $\Omega_0=\Omega_0(p,\mu,\kappa,P,a_0,u_0)$ and $c_0=c_0(p,\mu,\kappa,P,a_0,u_0)$ such that if $\Omega \in \mathbb{R} \setminus \{ 0\}$ and $\varepsilon > 0$ satisfies 
        \begin{align}\label{assumption}
            \Omega_0 \leq |\Omega| \leq c_0 \frac{1}{\varepsilon},
        \end{align}
        then the system
        \eqref{eq:CNSKC-2} admits a unique global solution $(a,m)$ in the class
        \begin{align}
           a,\varepsilon\nabla a, m \in{}& \tC ( [0,\infty) ; \fB_{p,1}^{\frac{3}{p}-3} ( \mathbb{R}^3 ) \cap \fB_{p,1}^{\frac{3}{p}-1} ( \mathbb{R}^3 ))
            \cap 
            L^1( 0,\infty ; \fB_{p,1}^{\frac{3}{p}+1} (\mathbb{R}^3)) 
        \end{align}
        with $\rho(t,x)= 1 + \varepsilon a(t,x)> 0$.
    \end{thm}
\begin{rem}
    Let us mention some remarks on the auxiliary regularity $\fB_{p,1}^{\frac{3}{p}-3} ( \mathbb{R}^3)$.
    \begin{itemize}
        \item
        From the linear analysis\footnote{See Lemma \ref{energylemma} and the estimate \eqref{hkjsll}},
        we see that the linear solutions behave as the semigroup $\big\{e^{-t\frac{\Delta^2}{\Omega^2\varepsilon^2-\Delta}}\big\}_{t>0}$, which acts as a forth-order heat kernel in the low frequency part.
        Let us consider this semigroup in maximal $L^1$-regularity $L^1(0,\infty;\fB_{p,1}^{\frac{3}{p}+1}(\mathbb{R}^3))$.
        Then, we have 
        \begin{align}
            \n{e^{-t\frac{\Delta^2}{\Omega^2\varepsilon^2-\Delta}}f}_{{L^{1}}(0,T;\widehat{\dot{B}}{}^{\frac{3}{p}+1}_{p,1})}
            \leq 
            C
            (\Omega\varepsilon)^{2}
            \|f\|_{\widehat{\dot{B}}{}^{\frac{3}{p}-3}_{p,1}}^{\ell;|\Omega|\varepsilon}
            +
            C
            \|f\|_{\widehat{\dot{B}}{}^{\frac{3}{p}-1}_{p,1}}^{h;|\Omega|\varepsilon},
        \end{align}
        where the truncated semi-norms appearing in the right-hand side above are defined in Section \ref{sec:pre}.
        Hence, we require the regularity $\widehat{\dot{B}}{}^{\frac{3}{p}-3}_{p,1}(\mathbb{R}^3)$ for the low frequency part of the initial data.
        \item 
        In the no-Korteweg case \cite{fujii2024global}, the corresponding auxiliary space was taken as 
        $\dB_{2,\infty}^{-\frac{3}{2}}(\mathbb{R}^3)$. 
        There, the third index of the Besov space had to be chosen as $\infty$ for technical reasons, 
        which prevented us from obtaining the global maximal $L^1$-regularity 
        $L^1(0,\infty;\dB_{2,1}^{\frac{5}{2}}(\mathbb{R}^3))$. 
        This restriction required a rather complicated argument in proving the global well-posedness.
        In contrast, in the present Korteweg case, one could similarly take the third index as $\infty$, 
        namely $\fB_{p,\infty}^{\frac{3}{p}-3}(\mathbb{R}^3)$, if one followed the complicated arguments 
        developed in the previous paper. 
        However, since the Korteweg system provides better regularity for the density perturbation, 
        we are able to choose the third index of the auxiliary space as $1$. 
        This choice yields the maximal $L^1$-regularity 
        $L^1(0,\infty;\fB_{p,1}^{\frac{3}{p}+1}(\mathbb{R}^3))$, 
        and allows us to present a comparatively simpler proof of the global a priori estimates.
    \end{itemize}
\end{rem}
This paper is organized as follows.
In Section \ref{sec:pre}, we recall the definitions and elementary properties of function spaces used in this paper.
In Section \ref{sec:lin}, we prepare some linear estimates.
In Section \ref{sec:pf}, making use of them, we establish the nonlinear global a priori estimates of the local solutions, which are given in the Appendix \ref{sec:A}, and complete the proof of our main result.

Throughout this paper, we denote by $C$ the constant, which may differ in each line. In particular, $C=C(a_1,...,a_n)$ means that $C$ depends only on $a_1,...,a_n$. 
For any integrability exponent $1\leq p \leq \infty$, we denote by $p'$ the H\"{o}lder conjugate of $p$. 
For two Banach spaces $X$ and $Y$ with $X \cap Y \neq \varnothing$, 
we write $\| \cdot \|_{X \cap Y} := \| \cdot \|_X + \| \cdot \|_Y$.
For $f=(f_1,...,f_n) \in X^n$, we set $\| f \|_{X}:= \| f_1\|_X + ... + \| f_n \|_X$.

\section{Preliminaries}\label{sec:pre}
In this section, we briefly introduce the Littlewood-Paley decomposition, some needed function spaces, and some related analysis tools. More details can be found in \cites{MR2768550,MR3815247,MR3839617}.\par 
We define $\mathscr{S}(\mathbb{R}^3)$ as the set of all Schwartz functions on $\mathbb{R}^3$ and $\mathscr{S}^\prime(\mathbb{R}^3)$ as the set of all tempered distributions on $\mathbb{R}^3$. For any $f\in \mathscr{S}(\mathbb{R}^3) $, its Fourier transform $\mathscr{F}[f]$ and inverse Fourier transform $\mathscr{F}^{-1}[f]$ are defined as
\begin{align}
\mathscr{F}[f](\xi) =\hf(\xi)
:=\int_{\mathbb{R}^3} e^{-i x \cdot \xi} f(x) dx,
\quad  
\mathscr{F}^{-1}[f](x) :=(2\pi)^{-3} 
\int_{\mathbb{R}^3} e^{i x \cdot \xi} f(\xi) d\xi.
\end{align}
For any $1\leq p \leq \infty$, we denote by $L^p(\mathbb{R}^3)$ the usual Lebesgue space and by $\widehat{L^p}(\mathbb{R}^3)$ the Fourier--Lebesgue space, defined as
\begin{align}
\widehat{L^p}(\mathbb{R}^3):=
\left\{
f\in \mathscr{S}'(\mathbb{R}^3); \hf \in L^{1}_{\rm loc}(\mathbb{R}^3),\,\, 
\| f \|_{\widehat{L^p}}:=\| \hf \|_{L^{p'}} <\infty
\right\}. 
\end{align}
Choosing a sequence of functions $\{\phi_j \}_{j\in \mathbb{Z}} \subset \mathscr{S}(\mathbb{R}^3)$ satisfying
\begin{align}
 & \phi_j (x)=2^{2j}\phi_0(2^j x),
 \quad 0\leq \widehat{\phi_0}    \leq 1, \\
 & \widehat{\phi_0}(\xi) \text{ is a radial function},\\
 &\rm{supp} \,\widehat{\phi_0}  \subset
 \left\{ 
 \xi \in \mathbb{R}^3; \frac{3}{4}\leq |\xi| \leq \frac{8}{3}
 \right\}, \\
 & \sum_{j\in \mathbb{Z}} 
 \widehat{\phi_j}(\xi)=1\,\, \text{ for all  }
 \xi \in \mathbb{R}^3 \backslash \{0\}. 
\end{align}
For any $j\in \mathbb{Z}$ and $f \in \mathscr{S}'(\mathbb{R}^2)$, the dyadic operator $\dot{\Delta}_j$ and low frequency cut-off operator $\dot{S}_j$ are defined as follows: 
\begin{align}
\Delta_j f:= \mathscr{F}^{-1}[  \widehat{\phi_j}  \hf],
\qquad S_j f := \sum_{j'\leq j-1} \Delta_{j'}f. 
\end{align}
We denote by $\mathscr{P}(\mathbb{R}^3)$ the set of all polynomials on $\mathbb{R}^3$. The following so-called Littlewood-Paley decomposition holds for all $f \in \mathscr{S}'(\mathbb{R}^3)/ \mathscr{P}(\mathbb{R}^3)$:\begin{equation*}\begin{split}
		f=\sum_{j\in\mathbb{Z}}\dot{\Delta}_j f.
	\end{split}
\end{equation*}Moreover, one may check that\begin{equation}\label{orthogon}
	\dot{\Delta}_k\dot{\Delta}_j f=0\quad\text{if}\quad|k-j|\geq2,\quad\dot{\Delta}_k(\dot{S}_{j-1}f\dot{\Delta}_jg)=0\quad\text{if}\quad|k-j|\geq5.
\end{equation}\par
Let us recall the following well-known Bernstein inequalities.
\begin{lemm}\label{berstein}
	Let $0<r<R$, $k\in\mathbb{N}$ and $1 \leq p \leq q \leq \infty$. Then, there exists a positive constant $C=C(r,R)$ such that \begin{equation}
		\begin{split}
			&\operatorname{supp}\widehat{f}\subset\{\xi\in\mathbb{R}^3;|\xi|\leq\lambda R \}\Longrightarrow\|\nabla^k f\|_{L^q}\leq C\lambda^{k+3(\frac{1}{p}-\frac{1}{q})}\|f\|_{L^p},\\
			&\operatorname{supp}\widehat{f}\subset\{\xi\in\mathbb{R}^3;|\xi|\leq\lambda R \}\Longrightarrow\|\nabla^k f\|_{\widehat{L^q}}\leq C\lambda^{k+3(\frac{1}{p}-\frac{1}{q})}\|f\|_{\widehat{L^p}},\\
			&\operatorname{supp}\widehat{f}\subset\{\xi\in\mathbb{R}^3;\lambda r\leq|\xi|\leq\lambda R \}\Longrightarrow 	\|f\|_{L^p}\leq C\lambda^{-k}\|\nabla^k f\|_{L^p},
			\\
			&\operatorname{supp}\widehat{f}\subset\{\xi\in\mathbb{R}^3;\lambda r\leq|\xi|\leq\lambda R \}\Longrightarrow 	\|f\|_{\widehat{L^p}}\leq C\lambda^{-k}\|\nabla^k f\|_{\widehat{L^p}}.
		\end{split}
	\end{equation}
\end{lemm}
Let us give the definitions of the homogeneous Besov space and the homogeneous Fourier--Besov space.\begin{df}\label{deff1}
	For $s\in\mathbb{R}$ and $1\leq p,\sigma\leq\infty$, the homogeneous Besov space $\dot{B}^{s}_{p,\sigma}(\mathbb{R}^3)$ is defined by\begin{align}
		\dot{B}^{s}_{p,\sigma}(\mathbb{R}^3):=\left\{f\in\mathscr{S}'(\mathbb{R}^3)/ \mathscr{P}(\mathbb{R}^3):\|f\|_{\dot{B}^s_{p,\sigma}}<\infty\right\}
	\end{align}with\begin{align}
		\|f\|_{\dot{B}^s_{p,\sigma}}:=\left\|\left\{2^{js}\|\dot{\Delta}_j f\|_{L^p}\right\}_{j\in\mathbb{Z}}\right\|_{\ell^\sigma}.
	\end{align}
\end{df}
\begin{df}\label{deff2}
	For $s\in\mathbb{R}$ and $1\leq p,\sigma\leq\infty$, the homogeneous Fourier--Besov space $\fB^{s}_{p,\sigma}(\mathbb{R}^3)$ is defined by\begin{align}
	\fB^s_{p,\sigma} (\mathbb{R}^3):=
	\left\{
	f \in \mathscr{S}'(\mathbb{R}^3)/ \mathscr{P}(\mathbb{R}^3); 
	\| f\|_{\fB^s_{p,\sigma}} <\infty
	\right\}
\end{align}with\begin{align}
	\| f\|_{\fB^s_{p,\sigma}}:= \left\|\left\{2^{js}\|\dot{\Delta}_j f\|_{\widehat{L^p}}\right\}_{j\in\mathbb{Z}}\right\|_{\ell^\sigma}
	=  \left\|   
	\left\{    2^{js} 
	\| \widehat{\phi_j} \hf \|
	_{L^{p'}} \right \}_{j\in \mathbb{Z}}
	\right\|_{\ell^{\sigma}}. 
\end{align}
\end{df} 
\begin{rem}
By the Plancherel theorem, one may see that $\fB^s_{2,\sigma} (\mathbb{R}^3)=\dot{B}^{s}_{2,\sigma}(\mathbb{R}^3)$. By the Hanusdorff-Young inequality, we also see that $\dot{B}^{s}_{p,\sigma}(\mathbb{R}^3)\hookrightarrow\fB^s_{p,\sigma} (\mathbb{R}^3)$ for $1\leq p<2$ and $\fB^s_{p,\sigma} (\mathbb{R}^3)\hookrightarrow\dot{B}^{s}_{p,\sigma}(\mathbb{R}^3)$ for $2<p\leq\infty$.
\end{rem}
As a direct consequence of \lemref{berstein}, \defref{deff1} and \defref{deff2}, we collect some properties on Besov spaces and Fourier--Besov spaces in the following lemma.\begin{lemm}\label{prop} Let $s\in\mathbb{R}$ and $1\leq p,\sigma\leq \infty$. Then the following properties hold true:
	\begin{enumerate} 
		\item[(i)] Derivatives: For $k\in\mathbb{N}$, it holds that\begin{equation*}
			\|\nabla^k f\|_{\dot{{B}}^{s}_{p,\sigma}}
			\sim
			\|f\|_{\dot{{B}}^{s+k}_{p,\sigma}},\quad 	\|\nabla^k f\|_{\widehat{\dot{{B}}}{}^{s}_{p,\sigma}}
			\sim
			\|f\|_{\widehat{\dot{{B}}}{}^{s+k}_{p,\sigma}}.
		\end{equation*}
		\item[(ii)] Embeddings: For $1\leq p_1\leq p_2 \leq \infty$ and $1\leq\sigma_1\leq \sigma_2 \leq \infty$, it holds that\begin{equation*}
			\begin{split}
				&\dot{{B}}^{s}_{p_1,\sigma_1}(\mathbb{R}^3)\hookrightarrow \dot{{B}}^{s-3(\frac{1}{p_1}-\frac{1}{p_2})}_{p_2,\sigma_2}(\mathbb{R}^3),\quad \fB^s_{p_1,\sigma_1} (\mathbb{R}^3) 
				\hookrightarrow
				\fB^{s-3(\frac{1}{p_1}-\frac{1}{p_2})}_{p_2,\sigma_2} (\mathbb{R}^3),\\
				& \dot{{B}}^{\frac{3}{p}}_{p,1}(\mathbb{R}^n)\hookrightarrow L^\infty(\mathbb{R}^3),\quad \widehat{\dot{{B}}}^{\frac{3}{p}}_{p,1}(\mathbb{R}^n)\hookrightarrow \widehat{L^\infty}(\mathbb{R}^3). 
			\end{split}
		\end{equation*}
		\item[(iii)] Interpolation: For $s_1<s_2$ and $0<\theta<1$, it holds that\begin{equation}\label{inter2}
			\begin{split}
			\|f\|_{{\dot{B}}^{\theta s_1+(1-\theta)s_2}_{p,\sigma}}&\leq \|f\|^\theta_{{\dot{B}}^{s_1}_{p,\sigma}}\|f\|^{1-\theta}_{{\dot{B}}^{s_2}_{p,\sigma}},\\	
                \|f\|_{\widehat{\dot{B}}{}^{\theta s_1+(1-\theta)s_2}_{p,\sigma}}&\leq \|f\|^\theta_{\widehat{\dot{B}}{}^{s_1}_{p,\sigma}}\|f\|^{1-\theta}_{\widehat{\dot{B}}{}^{s_2}_{p,\sigma}}.
			\end{split}
		\end{equation}
	\end{enumerate}		
\end{lemm}
To estimate the time-dependent functions, we need the following Chemin--Lerner space, which was initiated in \cite{MR1354312}.
\begin{df}
	For $s\in\mathbb{R}$, $1\leq p,r,\sigma\leq\infty$ and $I\subset\mathbb{R}$ be an interval, the space $\widetilde{L^r}(I;\dot{B}^{s}_{p,\sigma}(\mathbb{R}^3))$ is defined by
    \begin{equation*}
		\begin{split}
            \widetilde{L^r}(I;\dot{B}^{s}_{p,\sigma}(\mathbb{R}^3))
            :=
			\left\{
			f :  I \rightarrow \mathscr{S}'(\mathbb{R}^3)/ \mathscr{P}(\mathbb{R}^3); 
			\| f\|_{\widetilde{L^r}(I;\dot{B}^{s}_{p,\sigma})} <\infty
			\right\}
		\end{split}
	\end{equation*}
    with
    \begin{align}
		\| f\|_{\widetilde{L^r}(I;\dot{B}^{s}_{p,\sigma})}:= 
		\Big\|   
		\{    2^{js} \| \Delta_j f \|_{ L^r(I; {L^p} )    }  \}_{j\in \mathbb{Z}}
		\Big \|_{\ell^{\sigma}}
		. 
	\end{align}
\end{df} 
Similarly, in the setting of the Fourier--Besov space, we also define the space--time space $\widetilde{L^r}(I;\fB^{s}_{p,\sigma}(\mathbb{R}^3))$ as follows.
\begin{df}\label{clspace}
	For $s\in\mathbb{R}$, $1\leq p,r,\sigma\leq\infty$ and $I\subset\mathbb{R}$ be an interval, the space $\widetilde{L^r}(I;\fB^{s}_{p,\sigma}(\mathbb{R}^3))$ is defined by
    \begin{equation*}
		\begin{split}
            \widetilde{L^r}(I;\fB^{s}_{p,\sigma}(\mathbb{R}^3))
            :=
			\left\{
			f :  I \rightarrow \mathscr{S}'(\mathbb{R}^3)/ \mathscr{P}(\mathbb{R}^3); 
			\| f\|_{\widetilde{L^r}(I;\fB^{s}_{p,\sigma})} <\infty
			\right\}
		\end{split}
	\end{equation*}with\begin{align}
	\| f\|_{\widetilde{L^r}(I;\fB^{s}_{p,\sigma})}:= 
	\Big\|   
	\{    2^{js} \| \Delta_j f \|_{ L^r(I; \widehat{L^p} )    }  \}_{j\in \mathbb{Z}}
	\Big \|_{\ell^{\sigma}}
	=\Big\|   
	\{    2^{js} \| \widehat{\phi_j} \hf  \|_{ L^r(I; L^{p'} )    }  \}_{j\in \mathbb{Z}}
	\Big \|_{\ell^{\sigma}}. 
	\end{align}
\end{df} 
For notational simplicity, we set
\begin{align}
	\widetilde{C}(I;\fB^s_{p,\sigma}(\mathbb{R}^3))
	:=
	C(I;\fB^s_{p,\sigma}(\mathbb{R}^3))
	\cap
	\widetilde{L^{\infty}}
	(I;\fB^s_{p,\sigma}(\mathbb{R}^3))
\end{align}
and
\begin{align}
	\widetilde{C}(I;\fB^{s_1}_{p_1,\sigma_1}(\mathbb{R}^3) \cap \fB^{s_2}_{p_2,\sigma_2}(\mathbb{R}^3))
	:=
	\widetilde{C}(I;\fB^{s_1}_{p_1,\sigma_1}(\mathbb{R}^3))
	\cap
	\widetilde{C}(I;\fB^{s_2}_{p_2,\sigma_2}(\mathbb{R}^3))
\end{align}
for $1 \leq p,p_1,p_2 \leq \infty$, $1 \leq \sigma,\sigma_1,\sigma_2 \leq \infty$, and $s,s_1,s_2 \in \mathbb{R}$.
\begin{rem}\label{minkow}
	By the Minkowski inequality, one easily observes that
    \begin{align}
        &
        \|f\|_{\widetilde{L^r}(I;\widehat{\dot{B}}{}^s_{p,\sigma})}\leq\|f\|_{{L^r}(I;\widehat{\dot{B}}{}^s_{p,\sigma})}\qquad \text{if}\quad r\leq \sigma,  \\
		&
        \|f\|_{{L^r}(I;\widehat{\dot{B}}{}^s_{p,\sigma})}\leq\|f\|_{\widetilde{L^r}(I;\widehat{\dot{B}}{}^s_{p,\sigma})}\qquad \text{if}\quad r\geq \sigma.
    \end{align}
\end{rem}
For any $0<\alpha<\beta$, the truncated Fourier--Besov norms on the low, middle and high frequencies are defined as follows: 
\begin{align}
 \| f\|_{\fB^s_{p,\sigma}}^{\ell;\alpha}  &:=
 \Big\|   
 \{    2^{js} 
 \| \Delta_j f \|
 _{\widehat{L^p}}  \}_{2^{j} \leq \alpha}
 \Big \|_{\ell^{\sigma}}, \\
\| f\|_{\fB^s_{p,\sigma}}^{m;\alpha,\beta}  &:=
 \Big\|   
\{    2^{js} 
\| \Delta_j f \|
_{\widehat{L^p}}  \}_{ \alpha <  2^{j} \leq \beta}
\Big \|_{\ell^{\sigma}},  \\ 
 \| f\|_{\fB^s_{p,\sigma}}^{h;\beta}  &:=
\Big\|   
\{    2^{js} 
\| \Delta_j f \|
_{\widehat{L^p}}  \}_{2^{j} > \beta}
\Big \|_{\ell^{\sigma}}, \\ 
\| f\|_{\widetilde{L^r}(I;\fB^{s}_{p,\sigma})}^{\ell;\alpha} & := 
\Big\|   
\{    2^{js} \| \Delta_j f \|_{ L^r(I; \widehat{L^p} )    }  \}_{2^{j} \leq\alpha}
\Big \|_{\ell^{\sigma}} ,\\ 
\| f\|_{\widetilde{L^r}(I;\fB^{s}_{p,\sigma})}^{m;\alpha,\beta} & := 
\Big\|   
\{    2^{js} \| \Delta_j f \|_{ L^r(I; \widehat{L^p} )    }  \}_{ \alpha <2^{j} \leq \beta}
\Big \|_{\ell^{\sigma}},\\
\| f\|_{\widetilde{L^r}(I;\fB^{s}_{p,\sigma})}^{h;\beta} & := 
\Big\|   
\{    2^{js} \| \Delta_j f \|_{ L^r(I; \widehat{L^p} )    }  \}_{2^{j} >\beta}
\Big \|_{\ell^{\sigma}}. 
\end{align}
To estimate the nonlinear terms in Besov spaces and Fourier--Besov spaces, we shall use the following Bony decomposition:\begin{equation}\label{BD}
	fg=\dot{T}_fg+\dot{T}_gf+\dot{R}(f,g),
\end{equation}where \begin{equation}\label{trdef}
	\dot{T}_fg:=\sum_{j\in\mathbb{Z}}\dot{S}_{j-1}f\dot{\Delta}_jg\quad\text{and}\quad\dot{R}(f,g):=\sum_{j\in\mathbb{Z}}\sum_{|j^\prime-j|\leq 1}\dot{\Delta}_jf{\dot{\Delta}}_{j^\prime}g.
\end{equation}
\par Let us present the basic estimates of the operators $\dot{T},\dot{R}$ in Besov spaces and Fourier--Besov spaces in the following lemma. 
\begin{lemm}\label{lemm-Tfg-Rfg}
The following statements hold true:
\begin{enumerate}
    \item { Let $1\leq p,p_1,p_2,\sigma \leq \infty$ and $s,s_1,s_2\in \mathbb{R}$ satisfy
    \begin{align}
 \frac{1}{p}=\frac{1}{p_1}+\frac{1}{p_2},\quad
  s=s_1+s_2, \quad
 s_1\leq 0.
    \end{align}
    Then, there exists a positive constant $C=C(s_1,s_2)$ such that
    \begin{align}
     \n{ \dot{T}_f g  }_{ \dot{B}^{s}_{p,\sigma} }
\leq C
  \n{ f   }_{ \dot{B}^{s_1}_{p_1,1} }
  \n{ g }_{ \dot{B}^{s_2}_{p_2,\sigma} },\quad \n{ \dot{T}_fg }_{ \fB^{s}_{p,\sigma} } 
  \leq C
  \n{ f   }_{ \fB^{s_1}_{p_1,1} }
  \n{ g }_{ \fB^{s_2}_{p_2,\sigma} }.
     \end{align}
    }
    \item{ Let $1\leq p,p_1,p_2,\sigma,\sigma_1,\sigma_2 \leq \infty$ and $s,s_1,s_2\in \mathbb{R}$ satisfy
    \begin{align}
 \frac{1}{p}=\frac{1}{p_1}+\frac{1}{p_2},\quad
  s=s_1+s_2>0, \quad
 \frac{1}{\sigma} \leq \frac{1}{\sigma_1}+\frac{1}{\sigma_2}      .
    \end{align}
     Then, there exists a positive constant $C=C(s_1,s_2)$ such that
     \begin{align}
     \n{ \dot{R} (f ,g)  }_{ \dot{B}^{s}_{p,\sigma} }
\leq C
  \n{ f   }_{ \dot{B}^{s_1}_{p_1,\sigma_1} }
  \n{ g }_{ \dot{B}^{s_2}_{p_2,\sigma_2} },\quad    \n{ \dot{R} (f ,g)  }_{ \fB^{s}_{p,\sigma} }
  \leq C 
  \n{ f   }_{ \fB^{s_1}_{p_1,\sigma_1} }
  \n{ g }_{ \fB^{s_2}_{p_2,\sigma_2} }.
     \end{align}
    }
\end{enumerate}
\end{lemm}
For the proof of \lemref{lemm-Tfg-Rfg}, we refer to \cite{MR2768550} for the Besov case, and \cite{MR4390809} for the Fourier--Besov case.\par 
Next, we establish some product estimates in Besov spaces and Fourier--Besov spaces, which play a crucial role in the analysis of the nonlinear terms.
\begin{lemm}\label{keyproduct1}
	Let $1\leq p\leq \infty$ and $s,s_1,s_2,s_3,s_4\in\mathbb{R}$ satisfy\begin{align}
			s=s_1+s_2=s_3+s_4,\quad s<\min\left\{3,\frac{6}{p}\right\},\quad s_1\geq 0, \quad s_3\geq0.
	\end{align}Then, there exists a positive constant $C=C(p,s,s_1,s_3)$ such that
	\begin{equation}
		\begin{split}
	\|fg\|_{\widehat{\dot{B}}{}^{\frac{3}{p}-s}_{p,1}}\leq C\|f\|_{\widehat{\dot{B}}{}^{\frac{3}{p}-s_1}_{p,1}}\|g\|_{\widehat{\dot{B}}{}^{\frac{3}{p}-s_2}_{p,1}}+C\|g\|_{\widehat{\dot{B}}{}^{\frac{3}{p}-s_3}_{p,1}}\|f\|_{\widehat{\dot{B}}{}^{\frac{3}{p}-s_4}_{p,1}}.
		\end{split}
	\end{equation}
\end{lemm}
\begin{proof}
	According to \lemref{lemm-Tfg-Rfg} and embeddings, there exists a positive constant $C=C(p,s,s_1,s_3)$ such that\begin{align}
		\n{ \dot{T}_fg }_{\widehat{\dot{B}}{}^{\frac{3}{p}-s}_{p,1}}&\leq C\n{ f   }_{ \fB^{-s_1}_{\infty,1} }
		\n{ g }_{ \fB^{\frac{3}{p}-s_2}_{p,1} }\leq C\|f\|_{\widehat{\dot{B}}{}^{\frac{3}{p}-s_1}_{p,1}}\|g\|_{\widehat{\dot{B}}{}^{\frac{3}{p}-s_2}_{p,1}},\\
			\n{ \dot{T}_gf }_{\widehat{\dot{B}}{}^{\frac{3}{p}-s}_{p,1}}&\leq C\n{ g   }_{ \fB^{-s_3}_{\infty,1} }
		\n{ f }_{ \fB^{\frac{3}{p}-s_4}_{p,1} }\leq C\|g\|_{\widehat{\dot{B}}{}^{\frac{3}{p}-s_3}_{p,1}}\|f\|_{\widehat{\dot{B}}{}^{\frac{3}{p}-s_4}_{p,1}},\\
		\n{ \dot{R} (f ,g)  }_{\widehat{\dot{B}}{}^{\frac{3}{p}-s}_{p,1}}&\leq 
        \begin{cases}
            \begin{aligned}
            C\n{ \dot{R} (f ,g)  }_{ \dot{B}^{3-s}_{1,1} }
            &\leq 
            C\|f\|_{\widehat{\dot{B}}{}^{\frac{3}{p}-s_1}_{p,1}}
            \|g\|_{\widehat{\dot{B}}{}^{\frac{3}{p^\prime}-s_2}_{p^\prime,1}}\\
            &\leq 
            C
            \|f\|_{\widehat{\dot{B}}{}^{\frac{3}{p}-s_1}_{p,1}}
            \|g\|_{\widehat{\dot{B}}{}^{\frac{3}{p}-s_2}_{p,1}}
            \end{aligned}
            & \text{if}\quad 1\leq p<2,
            \\
            C\n{ \dot{R} (f ,g)  }_{ \dot{B}^{\frac{6}{p}-s}_{\frac{p}{2},1} }\leq C\|f\|_{\widehat{\dot{B}}{}^{\frac{3}{p}-s_1}_{p,1}}\|g\|_{\widehat{\dot{B}}{}^{\frac{3}{p}-s_2}_{p,1}} & \text{if}\quad p\geq 2,
        \end{cases}
	\end{align}which together with the Bony decomposition of $fg$ completes the proof.
\end{proof}
\begin{lemm}\label{keyproduct2}
	Let $2\leq q\leq 4$ and $s,s_1,s_2,s_3,s_4\in\mathbb{R}$ satisfy\begin{equation}
		\begin{split}
			s=s_1+s_2=s_3+s_4,\quad s<\frac{6}{q},\quad \min\left\{s_1,s_3\right\}\geq\frac{3}{2}-\frac{3}{q}.
		\end{split}
	\end{equation}Then, there exists a positive constant $C=C(q,s,s_1,s_3)$ such that
	\begin{equation}
		\begin{split}
			\|fg\|_{\dot{B}^{\frac{3}{2}-s}_{2,1}}\leq C\|f\|_{\dot{B}^{\frac{3}{q}-s_1}_{q,1}}\|g\|_{\dot{B}^{\frac{3}{q}-s_2}_{q,1}}+ C\|g\|_{\dot{B}^{\frac{3}{q}-s_3}_{q,1}}\|f\|_{\dot{B}^{\frac{3}{q}-s_4}_{q,1}}.
		\end{split}
	\end{equation}
\end{lemm}	
\begin{proof}
Applying \lemref{lemm-Tfg-Rfg} and embeddings indicates that there exists a positive constant $C=C(q,s,s_1,s_3)$ such that\begin{align}
	\|\dot{T}_fg\|_{\dot{B}^{\frac{3}{2}-s}_{2,1}}&\leq C \|f\|_{\dot{B}^{\frac{3}{2}-\frac{3}{q}-s_1}_{\frac{2q}{q-2},1}}\|g\|_{\dot{B}^{\frac{3}{q}-s_2}_{q,1}}\leq C \|f\|_{\dot{B}^{\frac{3}{q}-s_1}_{q,1}}\|g\|_{\dot{B}^{\frac{3}{q}-s_2}_{q,1}},\\
\|\dot{T}_gf\|_{\dot{B}^{\frac{3}{2}-s}_{2,1}}&\leq C \|g\|_{\dot{B}^{\frac{3}{2}-\frac{3}{q}-s_3}_{\frac{2q}{q-2},1}}\|f\|_{\dot{B}^{\frac{3}{q}-s_4}_{q,1}}\leq C \|g\|_{\dot{B}^{\frac{3}{q}-s_3}_{q,1}}\|f\|_{\dot{B}^{\frac{3}{q}-s_4}_{q,1}},\\
\|\dot{R}(f,g)\|_{\dot{B}^{\frac{3}{2}-s}_{2,1}}&\leq C \|\dot{R}(f,g)\|_{\dot{B}^{\frac{6}{q}-s}_{\frac{q}{2},1}}\leq C\|f\|_{\dot{B}^{\frac{3}{q}-s_1}_{q,1}}\|g\|_{\dot{B}^{\frac{3}{q}-s_2}_{q,1}},
	\end{align}
	which together with
	the Bony decomposition of $fg$ completes the proof.
\end{proof}
\begin{lemm}\label{keyproduct3}
	Let $2\leq p\leq \infty$, $2\leq q\leq 4$, $s>-\frac{6}{q}$, $s_1\geq 0$, $s_2\geq\frac{3}{2}-\frac{3}{q}$ and $\beta>0$. Then, there exists a positive constant $C=C(p,q,s,s_1,s_2)$ such that\begin{align}
	\|fg\|_{\widehat{\dot{B}}{}^{\frac{3}{p}+s}_{p,1}}&\leq C\|f\|_{\widehat{\dot{B}}{}^{\frac{3}{p}-s_1}_{p,1}}\|g\|_{\widehat{\dot{B}}^{\frac{3}{p}+s+s_1}_{p,1}}+ C\|g\|_{\dot{B}^{\frac{3}{q}-s_2}_{q,1}}\|f\|_{\dot{B}^{\frac{3}{q}+s+s_2}_{q,1}},\label{product31}\\
\|fg\|_{\widehat{\dot{B}}^{\frac{3}{p}+s}_{p,1}}^{h;\beta}&\leq C\|f\|_{\widehat{\dot{B}}{}^{\frac{3}{p}-s_1}_{p,1}}\|g\|_{\widehat{\dot{B}}^{\frac{3}{p}+s+s_1}_{p,1}}^{h;\frac{\beta}{16}}+ C\|g\|_{\dot{B}^{\frac{3}{q}-s_2}_{q,1}}\|f\|_{\dot{B}^{\frac{3}{q}+s+s_2}_{q,1}}.\label{product32}
	\end{align}
\end{lemm}
\begin{proof}
	Resorting to \lemref{lemm-Tfg-Rfg} and embeddings, we get that there exists a positive constant $C=C(p,q,s,s_1,s_2)$ such that\begin{align}
		\|\dot{T}_fg\|_{\widehat{\dot{B}}^{\frac{3}{p}+s}_{p,1}}&\leq C\|f\|_{\widehat{\dot{B}}^{-s_1}_{\infty,1}}\|g\|_{\widehat{\dot{B}}^{\frac{3}{p}+s+s_1}_{p,1}}\leq C\|f\|_{\widehat{\dot{B}}{}^{\frac{3}{p}-s_1}_{p,1}}\|g\|_{\widehat{\dot{B}}^{\frac{3}{p}+s+s_1}_{p,1}},\\
		\|\dot{T}_gf\|_{\widehat{\dot{B}}^{\frac{3}{p}+s}_{p,1}}&\leq C\|\dot{T}_gf\|_{\dot{B}^{\frac{3}{2}+s}_{2,1}}\\
		&\leq C \|g\|_{\dot{B}^{\frac{3}{2}-\frac{3}{q}-s_2}_{\frac{2q}{q-2},1}}\|f\|_{\dot{B}^{\frac{3}{q}+s+s_2}_{q,1}}\leq C \|g\|_{\dot{B}^{\frac{3}{q}-s_2}_{q,1}}\|f\|_{\dot{B}^{\frac{3}{q}+s+s_2}_{q,1}},\\
		\|\dot{R}(f,g)\|_{\widehat{\dot{B}}^{\frac{3}{p}+s}_{p,1}}&\leq C \|\dot{R}(f,g)\|_{\dot{B}^{\frac{6}{q}+s}_{\frac{q}{2},1}}\leq C\|g\|_{\dot{B}^{\frac{3}{q}-s_2}_{q,1}}\|f\|_{\dot{B}^{\frac{3}{q}+s+s_2}_{q,1}},
	\end{align}
	which together with
	the Bony decomposition of $fg$ completes the proof of \eqref{product31}. To obtain \eqref{product32}, it remains to prove\begin{equation}
		\begin{split}
			\|\dot{T}_fg\|_{\widehat{\dot{B}}^{\frac{3}{p}+s}_{p,1}}^{h;\beta}\leq C\|f\|_{\widehat{\dot{B}}^{\frac{3}{p}-s_1}_{p,1}}\|g\|_{\widehat{\dot{B}}^{\frac{3}{p}+s+s_1}_{p,1}}^{h;\frac{\beta}{16}}.
		\end{split}
	\end{equation}
Notice that \eqref{orthogon} implies\begin{equation}
		\begin{split}
			\dot{\Delta}_j\dot{T}_fg=\sum_{|j-j^\prime|\leq 4}\dot{\Delta}_j(\dot{S}_{j^\prime-1}f\dot{\Delta}_{j^\prime} g).
		\end{split}
	\end{equation}
Hence, by using the convolution inequality, we have
\begin{align}
    \|\dot{\Delta}_j\dot{T}_fg\|_{\widehat{L^p}}
    &\leq \sum_{|j-j^\prime|\leq 4}\|\dot{\Delta}_j(\dot{S}_{j^\prime-1}f\dot{\Delta}_{j^\prime} g)\|_{\widehat{L^p}}
    \\&\leq \sum_{|j-j^\prime|\leq 4} \left\|\left(\mathscr{F}[\dot{S}_{j^\prime-1}f]*\mathscr{F}[\dot{\Delta}_{j^\prime}g]\right)(\xi)\right\|_{L^{p^\prime}}\\
    &\leq\sum_{|j-j^\prime|\leq 4}\|\mathscr{F}[\dot{S}_{j^\prime-1}f](\xi)\|_{L^{1}}\|\mathscr{F}[\dot{\Delta}_{j^\prime}g](\xi)\|_{L^{p^\prime}}
    \\
    &
    = \sum_{|j-j^\prime|\leq 4} \|\dot{S}_{j^\prime-1}f\|_{\widehat{L^{\infty}}}\|\dot{\Delta}_{j^\prime} g\|_{\widehat{L^p}}\\
    &\leq \sum_{|j-j^\prime|\leq 4} \sum_{j^{\prime\prime}\leq j^\prime-2} \|\dot{\Delta}_{j^{\prime\prime}}f\|_{\widehat{L^{\infty}}}\|\dot{\Delta}_{j^\prime} g\|_{\widehat{L^p}}\\
    &\leq \left(\sum_{j^{\prime\prime}\leq j+2} \|\dot{\Delta}_{j^{\prime\prime}}f\|_{\widehat{L^{\infty}}}\right)\left(\sum_{|j-j^\prime|\leq 4} \|\dot{\Delta}_{j^\prime} g\|_{\widehat{L^p}}\right).
\end{align}
As a result, thanks to the convolution inequality and the embeddings, we see that
\begin{equation}\label{haobajiajy1}
	\begin{split}
			\|\dot{T}_fg&\|_{\widehat{\dot{B}}^{\frac{3}{p}+s}_{p,1}}^{h;\beta}= \sum_{2^j>\beta} 2^{j\left(\frac{3}{p}+s\right)}\|\dot{\Delta}_j\dot{T}_fg\|_{\widehat{L^{p}}}\\
			&\leq \sum_{2^j>\beta}2^{j\left(\frac{3}{p}+s\right)}\left(\sum_{j^{\prime\prime}\leq j+2} \|\dot{\Delta}_{j^{\prime\prime}}f\|_{\widehat{L^{\infty}}}\right)\left(\sum_{|j-j^\prime|\leq 4} \|\dot{\Delta}_{j^\prime} g\|_{\widehat{L^p}}\right)\\
			&\leq C\left(\sup_{2^j>\beta}\sum_{j^{\prime\prime}\leq j+2}2^{(j-j^{\prime\prime})(-s_1)}2^{-j^{\prime\prime}s_1} \|\dot{\Delta}_{j^{\prime\prime}}f\|_{\widehat{L^{\infty}}}\right)\left(\sum_{2^{j^\prime}> \frac{\beta}{16}}2^{{j^\prime}\left(\frac{3}{p}+s+s_1\right)} \|\dot{\Delta}_{j^\prime} g\|_{\widehat{L^p}}\right)
			\\
			&\leq   C\|f\|_{\widehat{\dot{B}}{}^{-s_1}_{\infty,1}} \|g\|_{\widehat{\dot{B}}^{\frac{3}{p}+s+s_1}_{p,1}}^{h;\frac{\beta}{16}}
			\\
			&\leq C \|f\|_{\widehat{\dot{B}}^{\frac{3}{p}-s_1}_{p,1}}\|g\|_{\widehat{\dot{B}}^{\frac{3}{p}+s+s_1}_{p,1}}^{h;\frac{\beta}{16}}.
		\end{split}
	\end{equation}This means the proof of \eqref{product32} is also completed.
\end{proof}
We end this section with the following lemma on the composition estimates in the Fourier--Besov spaces.
\begin{lemm}[\cite{MR4836084}*{Lemma 2.6}]\label{lemm-comp} 
Let $1 \leq p \leq \infty$ and $s \in \mathbb{R}$ satisfy 
\begin{align}
    -\min 
    \Mp{\frac{3}{p},\frac{3}{p'}}
    <s\leq 
    \frac{3}{p}.
\end{align}
Let $F=F(s)$ be a $C^1$ function on $\mathbb{R}$. Assume that $F(0)=0$ and $F$ is real analytic at $s=0$; $R_0$ denotes its radius of convergence.
Then, there exist two positive constant $c=c(p,F,R_0)$ and $C=C(p,s,F,R_0)$ such that for any $u,v \in \fB_{p,1}^s(\mathbb{R}^3) \cap \fB_{p,1}^{\frac{3}{p}}(\mathbb{R}^3)$ satisfying $\n{u}_{\fB_{p,1}^{\frac{3}{p}}}, \n{v}_{\fB_{p,1}^{\frac{3}{p}}} \leq c$, it holds that
\begin{align}
    \n{F(u)}_{\fB_{p,1}^s}
    \leq
    C
    \n{u}_{\fB_{p,1}^s},
    \qquad
    \n{F(u)-F(v)}_{\fB_{p,1}^s}
    \leq
    C
    \n{u-v}_{\fB_{p,1}^s}.
\end{align}
\end{lemm}

\section{Linear analysis}\label{sec:lin}
This section is devoted to establishing some energy and dispersive estimates for the solutions to the following linearized system which stems from \eqref{eq:CNSKC-2}:\begin{equation}\label{eq:CNSKC-3}
	\begin{cases}
		\partial_t a+\dfrac{1}{\varepsilon} \div m=0,
		& t>0, x\in \mathbb{R}^3,\\\begin{aligned}
			\partial_t m-\mu\Delta m-(\mu&+\lambda)\nabla\div m+\Omega ( e_3 \times m )\\&+\dfrac{1}{\varepsilon} \nabla a-\kappa\varepsilon\nabla\Delta a=f,
		\end{aligned}
		& t>0, x\in \mathbb{R}^3,\\
		a(0,x)=a_0(x),\quad m(0,x)=m_0(x), & x\in \mathbb{R}^3,
	\end{cases}
\end{equation}where $f$ is a given time-dependent vector field. First, we prove the following lemma on the energy estimates of the solutions to \eqref{eq:CNSKC-3}.
		\begin{lemm}\label{energylemma}
	Let $1\leq p,\sigma,r\leq\infty$, $s\in\mathbb{R}$ and $0<T\leq \infty$. There exists a positive constant $C=C(\mu,\kappa)$, independent of $\Omega$ and $\varepsilon$, such that the solution $(a,m)$ to \eqref{eq:CNSKC-3} satisfies 
    \begin{align}
	&\begin{aligned}\label{energyes1}
	    \|(a,\varepsilon \nabla a,m)\|_{\widetilde{L^{r}}(0,T;\widehat{\dot{B}}{}^{s+\frac{4}{r}}_{p,\sigma})}^{\ell;|\Omega|\varepsilon}&\leq C|\Omega|^{\frac{2}{r}}\varepsilon^{\frac{2}{r}}\|(a_0,\varepsilon \nabla a_0,m_0)\|_{\widehat{\dot{B}}{}^{s}_{p,\sigma}}^{\ell;|\Omega|\varepsilon}\\&\qquad+C|\Omega|^{\frac{2}{r}}\varepsilon^{\frac{2}{r}}\|f\|_{\widetilde{L^{1}}(0,T;\widehat{\dot{B}}{}^{s}_{p,\sigma})}^{\ell;|\Omega|\varepsilon},
	\end{aligned}		\\
            &\|(a,\varepsilon \nabla a,m)\|_{\widetilde{L^{r}}(0,T;\widehat{\dot{B}}{}^{s+\frac{2}{r}}_{p,\sigma})}^{h;|\Omega|\varepsilon}\leq C\|(a_0,\varepsilon \nabla a_0,m_0)\|_{\widehat{\dot{B}}{}^{s}_{p,\sigma}}^{h;|\Omega|\varepsilon}+C\|f\|_{\widetilde{L^{1}}(0,T;\widehat{\dot{B}}{}^{s}_{p,\sigma})}^{h;|\Omega|\varepsilon}.\label{energyes2}
	\end{align}
\end{lemm}
\begin{proof}
	Applying the Fourier transform to \eqref{eq:CNSKC-3}, one has
    \begin{equation}\label{eq:CNSKC-4}
		\begin{cases}
			\partial_t \widehat{a}+\dfrac{1}{\varepsilon} i\xi\cdot \widehat{m}=0,\\
				\partial_t \widehat{m}+\mu|\xi|^2 \widehat{m}+(\mu+\lambda)\xi(\xi\cdot \widehat{m})+\Omega ( e_3 \times \widehat{m} )+\dfrac{1}{\varepsilon} i\xi \widehat{a}+\kappa|\xi|^2(\varepsilon i\xi\widehat{a})=\widehat{f},\\
			\widehat{a}(0,\xi)=\widehat{a_0}(\xi),\quad \widehat{m}(0,\xi)=\widehat{m_0}(\xi).
		\end{cases}
	\end{equation}
Let $\Re z$ and $\overline{z}$ denote the real part and the complex conjugate of a complex number $z$ respectively, and let $\langle \cdot,\cdot\rangle$ denote the inner product in $\mathbb{C}^3$ (that is $\langle Z,Y\rangle=z_1\overline{y_1}+z_2\overline{y_2}+z_3\overline{y_3}$ for $Z=(z_1,z_2,z_3),Y=(y_1,y_2,y_3)\in\mathbb{C}^3$).
Multiplying the first equation of \eqref{eq:CNSKC-4} by $\overline{\widehat{a}}$, taking the $\mathbb{C}^3$-inner product of the second equation with $\widehat{m}$, and then summing up the resulting equalities and taking the real part of the sum, we get
\begin{equation}\label{energy1}
	\begin{split}
		\frac{1}{2}\frac{d}{dt}\left(|\widehat{a}|^2+|\widehat{m}|^2\right)+\underline{\mu}|\xi|^2|\widehat{m}|^2+\kappa|\xi|^2\Re \langle\varepsilon i\xi\widehat{a},\widehat{m}\rangle\leq |\widehat{f}||\widehat{m}|,
	\end{split}
\end{equation}
where $\underline{\mu}:=\min\{\mu,1\}$ and we have used the fact that
\begin{align}
		&\Re \left( i(\xi\cdot \widehat{m})\overline{\widehat{a}}+\langle i\xi \widehat{a},\widehat{m}\rangle\right)=0,
        \\
        &\langle e_3 \times \widehat{m},\widehat{m}\rangle=0,\\
        &\mu|\xi|^2 |\widehat{m}|^2+(\mu+\lambda)|\xi\cdot \widehat{m}|^2\geq \underline{\mu}|\xi|^2|\widehat{m}|^2.
\end{align}
Multiplying the first equation of \eqref{eq:CNSKC-4} by $\varepsilon i\xi$ gives\begin{equation}\label{epixia}
	\begin{split}
		\varepsilon i\xi\partial_t \widehat{a}-(\xi\cdot \widehat{m})\xi=0.
	\end{split}
\end{equation}Taking the $\mathbb{C}^3$-inner product of \eqref{epixia} with $\varepsilon i\xi \widehat{a}$ and then taking the real part of the resulting equality, we obtain\begin{equation}\label{energy2}
	\begin{split}
		\frac{1}{2}\frac{d}{dt}|\varepsilon i\xi\widehat{a}|^2-|\xi|^2\Re \langle\varepsilon i\xi\widehat{a},\widehat{m}\rangle=0,
	\end{split}
\end{equation}where we have used the fact that\begin{equation}\label{jianchi123}
    \langle(\xi\cdot \widehat{m})\xi,\varepsilon i\xi \widehat{a}\rangle=|\xi|^2\langle\widehat{m},\varepsilon i\xi\widehat{a}\rangle.
\end{equation} 
Next, taking the $\mathbb{C}^3$-inner product of \eqref{epixia} and the second equation of \eqref{eq:CNSKC-4} with $\widehat{m}$ and $\varepsilon i\xi \widehat{a}$ respectively, we have\begin{equation}\label{epixia888}
	\begin{split}
		&\langle\varepsilon i\xi\partial_t \widehat{a},\widehat{m}\rangle-|\xi\cdot \widehat{m}|^2=0,\\
		&\langle\partial_t \widehat{m},\varepsilon i\xi \widehat{a}\rangle +|\xi|^2\langle\widehat{m},\varepsilon i\xi\widehat{a}\rangle+\Omega\varepsilon \langle e_3 \times \widehat{m},i\xi \widehat{a}\rangle\\
        &\qquad\qquad\qquad\qquad+|\xi|^2| \widehat{a}|^2+\kappa|\xi|^2|\varepsilon i\xi\widehat{a}|^2=\langle\widehat{f},\varepsilon i\xi \widehat{a}\rangle,
	\end{split}
\end{equation}where we have used \eqref{jianchi123} again.
Adding up the two equalities in \eqref{epixia888} and taking the real part of the sum leads to\begin{equation}\label{energy3}
	\begin{split}
		&\partial_t\Re \langle\varepsilon i\xi \widehat{a},\widehat{m}\rangle+|\xi|^2| \widehat{a}|^2+\kappa|\xi|^2|\varepsilon i\xi\widehat{a}|^2\\
		&\leq |\widehat{f}||\varepsilon i\xi \widehat{a}|+|\xi|^2|\widehat{m}||\varepsilon i\xi\widehat{a}|+|\Omega|\varepsilon |\xi||\widehat{m}||\widehat{a}|+|\xi\cdot m|^2.
	\end{split}
\end{equation}
Let $\eta:=\min\{\frac{1}{2},\frac{\kappa}{2},\frac{\kappa\underline{\mu}}{2},\frac{\underline{\mu}}{4},\sqrt{\kappa}\}>0$ and define
a positive valued function $\widehat{V}=\widehat{V}(t,\xi)$ by
\begin{equation}\label{Vdef}
	\begin{split}
        \widehat{V}^2=\widehat{V}^2(t,\xi)
        :=
        \left(
        \Omega\varepsilon^2+|\xi|^2
        \right)
        \left(
        |\widehat{a}|^2+|\widehat{m}|^2+\kappa|\varepsilon i\xi\widehat{a}|^2
        \right)
        +
        2\eta |\xi|^2\Re \langle\varepsilon i\xi \widehat{a},\widehat{m}\rangle.
	\end{split}
\end{equation}
Hence, it follows from \eqref{energy1}, \eqref{energy2} and \eqref{energy3} that\begin{equation}\label{energy4}
	\begin{split}
		&\frac{1}{2}\frac{d}{dt}\widehat{V}^2+\underline{\mu}\Omega^2\varepsilon^2|\xi|^2|\widehat{m}|^2+\underline{\mu}|\xi|^4|\widehat{m}|^2+\eta|\xi|^4|\widehat{a}|^2+\eta\kappa|\xi|^4|\varepsilon i\xi\widehat{a}|^2\\
		&\leq \left(\Omega^2\varepsilon^2+|\xi|^2\right)|\widehat{f}||\widehat{m}|+\eta |\xi|^2|\widehat{f}||\varepsilon i\xi \widehat{a}|+\eta|\xi|^4|\widehat{m}||\varepsilon i\xi\widehat{a}|\\&\qquad+\eta|\Omega|\varepsilon |\xi|^3|\widehat{m}||\widehat{a}|+\eta|\xi|^2|\xi\cdot m|^2.
	\end{split}
\end{equation}
Using the Young inequality, we find from \eqref{Vdef} that\begin{equation}
	\begin{split}
		&\left|\widehat{V}^2-\left(\Omega^2\varepsilon^2+|\xi|^2\right)|(\widehat{a},\sqrt{\kappa}\varepsilon i\xi\widehat{a},\widehat{m})|^2\right|\\
        &
        \quad\leq 
        2\eta |\xi|^2|\varepsilon i\xi\widehat{a}||\widehat{m}|
        \\
        &
        \quad\leq
        \eta|\xi|^2|(\widehat{m},\varepsilon i\xi\widehat{a})|^2
        \\
        &
        \quad\leq 
        \frac{1}{2}|\xi|^2|(\widehat{m},\sqrt{\kappa}\varepsilon i\xi\widehat{a})|^2,
	\end{split}
\end{equation}which implies 
\begin{equation}\label{energy5}
	\begin{split}
	\frac{1}{2}\left(\Omega^2\varepsilon^2+|\xi|^2\right)|(\widehat{a},\sqrt{\kappa}\varepsilon i\xi\widehat{a},\widehat{m})|^2	\leq\widehat{V}^2\leq\frac{3}{2}\left(\Omega^2\varepsilon^2+|\xi|^2\right)|(\widehat{a},\sqrt{\kappa}\varepsilon i\xi\widehat{a},\widehat{m})|^2.
	\end{split}
\end{equation}
Moreover, we have
\begin{equation}\label{energy6}
	\begin{split}
		\eta|\xi|^4|\widehat{m}||\varepsilon i\xi\widehat{a}|&\leq  \frac{\eta}{2\kappa}|\xi|^4|\widehat{m}|^2+\frac{\eta\kappa}{2}|\xi|^4|\varepsilon i\xi\widehat{a}|^2\\
		&\leq  \frac{\underline{\mu}}{4}|\xi|^4|\widehat{m}|^2+\frac{\eta\kappa}{2}|\xi|^4|\varepsilon i\xi\widehat{a}|^2,\\
		\eta|\Omega|\varepsilon |\xi|^3|\widehat{m}||\widehat{a}|&\leq \frac{\eta}{2}\Omega^2\varepsilon^2|\xi|^2|\widehat{m}|^2 +\frac{\eta}{2}|\xi|^4|\widehat{a}|^2\\
		&\leq \frac{\underline{\mu}}{8}\Omega^2\varepsilon^2|\xi|^2|\widehat{m}|^2 +\frac{\eta}{2}|\xi|^4|\widehat{a}|^2,\\
		\eta|\xi|^2|\xi\cdot m|^2&\leq \frac{\underline{\mu}}{4}|\xi|^4|\widehat{m}|^2.
	\end{split}
\end{equation}
As a result, we can deduce from \eqref{energy4}, \eqref{energy5} and \eqref{energy6} that\begin{equation}
	\begin{split}
		\frac{1}{2}\frac{d}{dt}\widehat{V}^2+\frac{\eta}{2}|\xi|^4|(\widehat{a},\sqrt{\kappa}\varepsilon i\xi\widehat{a},\widehat{m})|^2\leq \left(\Omega^2\varepsilon^2+|\xi|^2\right)|\widehat{f}||(\widehat{m},\sqrt{\kappa}\varepsilon i\xi\widehat{a})|,
	\end{split}
\end{equation}which leads to\begin{equation}
\begin{split}
\frac{1}{2}\frac{d}{dt}\widehat{V}^2+\frac{\eta}{3}\Theta(|\xi|,\Omega\varepsilon)\widehat{V}^2\leq 2\sqrt{\Omega^2\varepsilon^2+|\xi|^2}|\widehat{f}|\widehat{V},
\end{split}
\end{equation}
where\begin{equation}
	\begin{split}
		\Theta(|\xi|,\Omega\varepsilon):=\frac{|\xi|^4}{\Omega^2\varepsilon^2+|\xi|^2}.
	\end{split}
\end{equation}
Therefore, we arrive at\begin{equation}\label{energy7}
	\begin{split}
		\widehat{V}(t,\xi)\leq e^{-\frac{\eta}{3}\Theta(|\xi|,\Omega\varepsilon)t}\widehat{V}(0,\xi)+2\sqrt{\Omega^2\varepsilon^2+|\xi|^2}\int_{0}^{t}e^{-\frac{\eta}{3}\Theta(|\xi|,\Omega\varepsilon)(t-\tau)}|\widehat{f}(\tau,\xi)|d\tau.
	\end{split}
\end{equation}
Combining \eqref{energy5} and \eqref{energy7}, we obtain\begin{equation}\label{hkjsll}
	\begin{split}
		|(\widehat{a},\varepsilon i\xi\widehat{a},\widehat{m})(t,\xi)|&\leq Ce^{-\frac{\eta}{3}\Theta(|\xi|,\Omega\varepsilon)t}|(\widehat{a_0},\varepsilon i\xi\widehat{a_0},\widehat{m_0})(\xi)|\\
		&\quad+C\int_{0}^{t}e^{-\frac{\eta}{3}\Theta(|\xi|,\Omega\varepsilon)(t-\tau)}|\widehat{f}(\tau,\xi)|d\tau.
	\end{split}
\end{equation}
Multiplying the inequality above by $\widehat{\phi_j}(\xi)$ and taking the $L^r(0,T;L^{p^\prime}(\mathbb{R}^3_\xi))$-norm, we have\begin{equation}
	\begin{split}
		&\|(\dot{\Delta}_ja,\varepsilon\dot{\Delta}_j\nabla a,\dot{\Delta}_jm)\|_{L^r(0,T;\widehat{L^p})}\\
		&\leq C(\Omega^2\varepsilon^22^{-4j}+2^{-2j})^{\frac{1}{r}} \|(\dot{\Delta}_ja_0,\varepsilon\dot{\Delta}_j\nabla a_0,\dot{\Delta}_jm_0)\|_{\widehat{L^p}}\\
		  &\qquad+C(\Omega^2\varepsilon^22^{-4j}+2^{-2j})^{\frac{1}{r}}\|\dot{\Delta}_jf\|_{L^1(0,T;\widehat{L^p})}.
	\end{split}
\end{equation}
From the inequality above, we get\begin{equation}\label{cmxw1}
	\begin{split}
		&2^{j(s+\frac{4}{r})}\|(\dot{\Delta}_ja,\varepsilon\dot{\Delta}_j\nabla a,\dot{\Delta}_jm)\|_{L^r(0,T;\widehat{L^p})}\\
		&\leq C|\Omega|^{\frac{2}{r}}\varepsilon^{\frac{2}{r}}2^{js} \|(\dot{\Delta}_ja_0,\varepsilon\dot{\Delta}_j\nabla a_0,\dot{\Delta}_jm_0)\|_{\widehat{L^p}}+C|\Omega|^{\frac{2}{r}}\varepsilon^{\frac{2}{r}}2^{js}\|\dot{\Delta}_jf\|_{L^1(0,T;\widehat{L^p})}
	\end{split}
\end{equation}for all $j$ with $2^j\leq|\Omega|\varepsilon$, and \begin{equation}\label{cmxw2}
\begin{split}
	&2^{j(s+\frac{2}{r})}\|(\dot{\Delta}_ja,\varepsilon\dot{\Delta}_j\nabla a,\dot{\Delta}_jm)\|_{L^r(0,T;\widehat{L^p})}\\
	&\leq C2^{js} \|(\dot{\Delta}_ja_0,\varepsilon\dot{\Delta}_j\nabla a_0,\dot{\Delta}_jm_0)\|_{\widehat{L^p}}+C2^{js}\|\dot{\Delta}_jf\|_{L^1(0,T;\widehat{L^p})}
\end{split}
\end{equation}for all $j$ with $2^j>|\Omega|\varepsilon$. Thus, taking $\ell^\sigma$-norm of \eqref{cmxw1} and \eqref{cmxw2} on $\{j\in\mathbb{Z};\,2^j\leq|\Omega|\varepsilon\}$ and $\{j\in\mathbb{Z};\,2^j>|\Omega|\varepsilon\}$ respectively, we end up with the desired estimates \eqref{energyes1} and \eqref{energyes2}. This completes the proof.
\end{proof}
Next, we focus on the dispersive estimates for the high frequencies of the solutions to \eqref{eq:CNSKC-3}.
\begin{lemm}\label{dispersivelemma}
Let $0<T\leq \infty$, $s\in\mathbb{R}$, $1\leq\sigma\leq\infty$ and	$2\leq p,q,r\leq\infty$ satisfy 
    \begin{equation}
		\begin{split}
		2\leq p\leq q<\infty,\quad 0\leq 
			\frac{1}{r}\leq\frac{1}{p}-\frac{1}{q}.
		\end{split}
	\end{equation}
    There exists a positive constant $C=C(p,q,r,\mu,\kappa)$, independent of $\Omega$ and $\varepsilon$, such that the solution $(a,m)$ to \eqref{eq:CNSKC-3} satisfies
    \begin{equation}\label{lineardispersive2}
	\begin{split}
		\|(a,m)\|_{\widetilde{L^{r}}(0,T;\dot{B}^{\frac{3}{q}+s}_{q,\sigma})}^{h;|\Omega|\varepsilon}&\leq C|\Omega|^{-\frac{1}{r}}\left(\|(a_0,\varepsilon\nabla a_0,m_0)\|_{\widehat{\dot{B}}{}^{\frac{3}{p}+s}_{p,\sigma}}^{h;|\Omega|\varepsilon}+\|f\|_{\widetilde{L^1}(0,T;\widehat{\dot{B}}{}^{\frac{3}{p}+s}_{p,\sigma})}^{h;|\Omega|\varepsilon}\right).
	\end{split}
\end{equation}
\end{lemm}
\begin{proof}
	We first consider the following inviscid linear system:\begin{equation}
		\begin{cases}
			\partial_t b+\dfrac{1}{\varepsilon}\div v=0,\\
			\partial_t v+\Omega(e_3\times v)+\dfrac{1}{\varepsilon}\nabla b=0,\\
		{b}(0,x)={b_0}(x),\quad v(0,x)={v_0}(x).
		\end{cases}
	\end{equation}	
	Denote by $\Mp{U_{\Omega,\varepsilon}(t)}_{t>0}$ the semigroup associated with the system above, and then, its solution $(b(t),v(t))^\top=U_{\Omega,\varepsilon}(t)(b_0,v_0)^\top$.
	Following the proof of \eqref{energy1}, we obtain\begin{equation}
		\begin{split}
			\frac{1}{2}\frac{d}{dt}|(\widehat{b}(t,\xi),\widehat{v}(t,\xi))|^2=0,
		\end{split}
	\end{equation}which together with the Hausdorff-Young inequality implies\begin{equation}\label{cgzynld1}
	\begin{split}
	\|\dot{\Delta}_jU_{\Omega,\varepsilon}(t)( b_0,v_0)^{\top}\|_{L^\infty(0,T;L^q)}
    &\leq 	
    \|\widehat{\phi_j}(\xi)(\widehat{b}(t,\xi),\widehat{v}(t,\xi))\|_{L^\infty(0,T;L^{q^\prime})}\\
	&\leq 
    \|\widehat{\phi_j}(\xi)(\widehat{b_0}(\xi),\widehat{v_0}(\xi))\|_{L^{q^\prime}}
    \\
    &= 
    \|\dot{\Delta}_j(b_0,v_0)\|_{\widehat{L^q}}
	\end{split}
	\end{equation}for all $j\in\mathbb{Z}$. On the other hand, let $0\leq \theta\leq 1$ and $r_1\geq0$ satisfy\begin{equation}
	\begin{split}
		\frac{1}{p}=\frac{\theta}{2}+\frac{1-\theta}{q},\quad \frac{1}{r}=\frac{\theta}{r_1}+\frac{1-\theta}{\infty}.
	\end{split}
	\end{equation}
	Then, one may see that\begin{equation}
		\begin{split}
			2\leq q<\infty,\quad \frac{1}{q}+\frac{1}{r_1}\leq\frac{1}{2}.
		\end{split}
	\end{equation}
	It follows from \cite{MR4836084}*{Lemma 3.2} that there exists a positive constant $C=C(q,r_1)$ such that\begin{equation}\label{cgzynld2}
		\begin{split}
			\|\dot{\Delta}_jU_{\Omega,\varepsilon}(t)(b_0,v_0)^{\top}\|_{L^{r_1}(0,T;L^q)}\leq C2^{3(\frac{1}{2}-\frac{1}{q})j}|\Omega|^{-\frac{1}{r_1}}\|\dot{\Delta}_j({b}_0,{v}_0)\|_{L^2}
		\end{split}
	\end{equation}for all $j$ with $2^{j}>|\Omega|\varepsilon$. Therefore, applying the complex interpolation for \eqref{cgzynld1} and \eqref{cgzynld2} yields\begin{equation}\label{cmxw22}
	\begin{split}
		\|\dot{\Delta}_jU_{\Omega,\varepsilon}(t)(b_0,v_0)^{\top}\|_{L^r(0,T;L^q)}\leq C2^{3(\frac{1}{p}-\frac{1}{q})j}|\Omega|^{-\frac{1}{r}}\|\dot{\Delta}_j({b}_0,{v}_0)\|_{\widehat{L^p}}
	\end{split}
	\end{equation}for all $j$ with $2^{j}>|\Omega|\varepsilon$. Thus, multiplying both sides of \eqref{cmxw22} by $2^{(\frac{3}{q}+s)j}$ and then taking $\ell^\sigma$-norm over $\{j\in\mathbb{Z};\,2^j>|\Omega|\varepsilon\}$, we obtain the following Strichartz estimates for the initial data in Fourier--Besov spaces:
    \begin{equation}\label{stries}
	\begin{split}
		\|U_{\Omega,\varepsilon}(t)(b_0,v_0)^{\top}\|_{\widetilde{L^{r}}(0,T;\dot{B}^{\frac{3}{q}+s}_{q,\sigma})}^{h;|\Omega|\varepsilon}&\leq C|\Omega|^{-\frac{1}{r}}\|(b_0,v_0)\|_{\widehat{\dot{B}}{}^{\frac{3}{p}+s}_{p,\sigma}}^{h;|\Omega|\varepsilon}.
	\end{split}
	\end{equation}
	\par 
	Next, we rewrite \eqref{eq:CNSKC-3} into\begin{equation}
		\begin{cases}
			\partial_t a+\dfrac{1}{\varepsilon} \div m=0,\\
				\partial_t m+\Omega ( e_3 \times m )+\dfrac{1}{\varepsilon} \nabla a=f+\mu\Delta m+(\mu+\lambda)\nabla\div m+\kappa\varepsilon\nabla\Delta a,\\
			a(0,x)=a_0(x),\quad m(0,x)=m_0(x).
		\end{cases}
	\end{equation}
Then, by the Duhamel principle, we see that\begin{equation}\begin{split}
\binom{a(t)}{m(t)}&=U_{\Omega,\varepsilon}(t)\binom{a_0}{m_0}\\
&\quad+\int_{0}^{t}U_{\Omega,\varepsilon}(t-\tau)\binom{0}{f+\mu\Delta m+(\mu+\lambda)\nabla\div m+\kappa\varepsilon\nabla\Delta a}(\tau)d\tau.
		\end{split}
	\end{equation}
Hence, resorting to the Strichartz estimates \eqref{stries}, we get
\begin{equation}\label{gjhkcg1}
		\begin{split}
			\|(a,m)\|_{\widetilde{L^{r}}(0,T;\dot{B}^{\frac{3}{q}+s}_{q,\sigma})}^{h;|\Omega|\varepsilon}&\leq C|\Omega|^{-\frac{1}{r}}\left(\|(a_0,m_0)\|_{\widehat{\dot{B}}{}^{\frac{3}{p}+s}_{p,\sigma}}^{h;|\Omega|\varepsilon}+\|f\|_{\widetilde{L^1}(0,T;\widehat{\dot{B}}{}^{\frac{3}{p}+s}_{p,\sigma})}^{h;|\Omega|\varepsilon}\right.\\
			&\quad\left.+\|m\|_{\widetilde{L^1}(0,T;\widehat{\dot{B}}{}^{\frac{3}{p}+s+2}_{p,\sigma})}^{h;|\Omega|\varepsilon}+\|\varepsilon\nabla a\|_{\widetilde{L^1}(0,T;\widehat{\dot{B}}{}^{\frac{3}{p}+s+2}_{p,\sigma})}^{h;|\Omega|\varepsilon}\right).
		\end{split}
	\end{equation}
Recall that \eqref{energyes2} gives us\begin{equation}\label{gjhkcg2}
	\begin{split}
		\|(m,\varepsilon \nabla a)\|_{\widetilde{L^1}(0,T;\widehat{\dot{B}}{}^{\frac{3}{p}+s+2}_{p,\sigma})}^{h;|\Omega|\varepsilon}&\leq C\|(a_0,\varepsilon \nabla a_0,m_0)\|_{\widehat{\dot{B}}{}^{\frac{3}{p}+s}_{p,\sigma}}^{h;|\Omega|\varepsilon}+C\|f\|_{\widetilde{L^1}(0,T;\widehat{\dot{B}}{}^{\frac{3}{p}+s}_{p,\sigma})}^{h;|\Omega|\varepsilon}.
	\end{split}
\end{equation}
Finally, inserting \eqref{gjhkcg2} into \eqref{gjhkcg1} leads to \eqref{lineardispersive2}. This completes the proof.
\end{proof}

\section{Proof of \thmref{thm:large}}\label{sec:pf}	
The goal of this section is to prove \thmref{thm:large}. For the sake of clarity in presentation, we divide the proof into three steps; the first step provides the local well-posedness, the second step establishes the global a priori estimates and we complete the proof in the final step by the continuous argument via the a priori estimates. 
\\
\textbf{First step: Local well-posedness}\par
 We first provide the following local well-posedness result of \eqref{eq:CNSKC-2}.
\begin{prop}\label{localwp}
Let $1 \leq p <3$
and let
$a_0 \in \fB_{p,1}^{\frac{3}{p}-3} ( \mathbb{R}^3 ) \cap \fB_{p,1}^{\frac{3}{p}} ( \mathbb{R}^3 )$
and
$m_0 \in  \fB_{p,1}^{\frac{3}{p}-3} ( \mathbb{R}^3 ) \cap \fB_{p,1}^{\frac{3}{p}-1} ( \mathbb{R}^3 ) $. There exists a positive constant $\varepsilon_0=\varepsilon_0(p,\mu,\kappa,P,a_0,u_0)$ and a positive time $T_0=T_0(p,\mu,\kappa,P,a_0,u_0)$ such that if $
		0<\varepsilon\leq\varepsilon_0$ and $0<|\Omega|\leq 1/\varepsilon$, then the system \eqref{eq:CNSKC-2} admits a unique local solution $(a,m)$ in the class
\begin{align}\label{class123}
 a,\varepsilon\nabla a, m \in \tC ( [0,T_0] ; \fB_{p,1}^{\frac{3}{p}-3} ( \mathbb{R}^3 ) \cap \fB_{p,1}^{\frac{3}{p}-1} ( \mathbb{R}^3 ))
\cap 
L^1( 0,T_0; \fB_{p,1}^{\frac{3}{p}+1} (\mathbb{R}^3)) 
\end{align}
with the estimates\begin{equation}\label{boundnessofa}
    \n{\varepsilon a}_{L^{\infty}(0,T_0;L^{\infty})} \leq 1/2,\quad \|\varepsilon a\|_{L^\infty(0,T_0;\widehat{\dot{B}}{}^{\frac{3}{p}}_{p,1})}\leq c_1,
\end{equation}for some constant $0<c_1\leq 1$.
\end{prop}
The proof of \propref{localwp} is deferred to Appendix \ref{sec:A}. \par 
\noindent\textbf{Second step: A priori estimates}\par
Next, we establish some a priori estimates for the solutions to \eqref{eq:CNSKC-2} in the following two lemmas. 
To simplify the writing, we prepare some notations in the rest of this section.
Let $(a,m)$ be the solution to \eqref{eq:CNSKC-2} on the time interval $[0,T_{\rm max})$ with the initial data $a_0 \in \fB_{p,1}^{\frac{3}{p}-3}(\mathbb{R}^3) \cap \fB_{p,1}^{\frac{3}{p}}(\mathbb{R}^3)$ and $m_0 \in \fB_{p,1}^{\frac{3}{p}-3}(\mathbb{R}^3) \cap \fB_{p,1}^{\frac{3}{p}-1}(\mathbb{R}^3)$, constructed as in the above proposition.
Here, $T_{\rm max} \in (T_0,\infty]$ denotes the maximal existence time.
For the sake of simplicity, we define
\begin{align}
	\mathcal{E}_{p}(T)&:=\|(a,\varepsilon \nabla a,m)\|_{\widetilde{L^{\infty}}(0,T;\widehat{\dot{B}}{}^{\frac{3}{p}-3}_{p,1})\cap \widetilde{L^{\infty}}(0,T;\widehat{\dot{B}}{}^{\frac{3}{p}-1}_{p,1})\cap{L^{1}}(0,T;\widehat{\dot{B}}{}^{\frac{3}{p}+1}_{p,1})},
	\\
	\mathcal{D}_{p,q,r}(T)&:=\n{(a,\varepsilon\nabla a,m)}_{\widetilde{L^r}(0,T;\dB_{q,1}^{\frac{3}{q}-3+\frac{4}{r}})\cap \widetilde{L^r}(0,T;\dB_{q,1}^{\frac{3}{q}-1+\frac{2}{r}})}+\|\varepsilon  \nabla a\|_{\widetilde{L^\infty}(0,T;\widehat{\dot{B}}{}^{\frac{3}{p}-1}_{p,1})},\\
	\mathcal{E}_{p,0}&:=\|(a_0,\varepsilon \nabla a_0,m_0)\|_{\widehat{\dot{B}}{}^{\frac{3}{p}-1}_{p,1}\cap \widehat{\dot{B}}{}^{\frac{3}{p}-3}_{p,1}}
\end{align} 
for $T \in (0,T_{\rm max}]$.
\begin{lemm}\label{xygj1} 
Let $2\leq p<q<3$ and $2<r<\infty$ satisfy\begin{equation}
	\begin{split}
\frac{1}{r}\leq\frac{3}{2q}-\frac{1}{4}.
	\end{split}
\end{equation}
Then, there exists a constant $C_1=C_1(p,q,r,\mu,\kappa,P) \geq 1$ such that for any $\Omega\in\mathbb{R}\setminus\{0\}$ and $0 < \varepsilon \leq \varepsilon_0$ with $|\Omega|\varepsilon\leq 1$, the solution $(a,m)$ to \eqref{eq:CNSKC-2} enjoys 
\begin{equation}\label{firstapriories}
\begin{split}
	\mathcal{E}_p(T)\leq C_1\mathcal{E}_{p,0}+C_1\left(1+\mathcal{D}_{p,q,r}(T)\right)\mathcal{D}_{p,q,r}(T)\mathcal{E}_p(T)
\end{split}
\end{equation}
for all $T \in (0,T_{\rm max}]$.
\end{lemm}
\begin{proof}
First of all, we state that the constant $c_1$ in \eqref{boundnessofa} is small enough that we can freely use the composition estimates stated in \lemref{lemm-comp} throughout this section; see the proof of \propref{localwp}. Applying \lemref{energylemma} to \eqref{eq:CNSKC-2} and noticing that $|\Omega|\varepsilon\leq 1$, we have\begin{align}
\mathcal{E}_p(T)&=\|(a,\varepsilon \nabla a,m)\|_{\widetilde{L^{\infty}}(0,T;\widehat{\dot{B}}{}^{\frac{3}{p}-3}_{p,1})\cap \widetilde{L^{\infty}}(0,T;\widehat{\dot{B}}{}^{\frac{3}{p}-1}_{p,1})\cap{L^{1}}(0,T;\widehat{\dot{B}}{}^{\frac{3}{p}+1}_{p,1})}^{\ell;|\Omega|\varepsilon}\\
&\qquad+\|(a,\varepsilon \nabla a,m)\|_{\widetilde{L^{\infty}}(0,T;\widehat{\dot{B}}{}^{\frac{3}{p}-3}_{p,1})\cap \widetilde{L^{\infty}}(0,T;\widehat{\dot{B}}{}^{\frac{3}{p}-1}_{p,1})\cap{L^{1}}(0,T;\widehat{\dot{B}}{}^{\frac{3}{p}+1}_{p,1})}^{h;|\Omega|\varepsilon}\label{zhuyaoes1}\\
&\leq C\|(a_0,\varepsilon \nabla a_0,m_0)\|_{\widehat{\dot{B}}{}^{\frac{3}{p}-1}_{p,1}\cap \widehat{\dot{B}}{}^{\frac{3}{p}-3}_{p,1}}+C\|N_\varepsilon[a,m]\|_{{L^{1}}(0,T;\widehat{\dot{B}}{}^{\frac{3}{p}-3}_{p,1})\cap{L^{1}}(0,T;\widehat{\dot{B}}{}^{\frac{3}{p}-1}_{p,1})}.
\end{align}
Next, we turn to tackle the nonlinear terms in the right-hand side above. By virtue of the embedding relation:\begin{align}
{\dot{B}}{}^{\frac{3}{2}+s}_{2,1}(\mathbb{R}^3)\hookrightarrow\widehat{\dot{B}}{}^{\frac{3}{p}+s}_{p,1}(\mathbb{R}^3)\quad \text{for all } s\in\mathbb{R},\label{needembedding}
\end{align}and \lemref{keyproduct2} (with $(s,s_1,s_3)=(2,3-\frac{4}{r},3-\frac{4}{r})$ and $(s,s_1,s_3)=(0,1-\frac{2}{r},1-\frac{2}{r})$), we get\begin{align}
	\|\div(m\otimes m)\|_{{L^1(0,T;\widehat{\dot{B}}{}^{\frac{3}{p}-3}_{p,1})}}&\leq C\|m\otimes m\|_{{L^1(0,T;\widehat{\dot{B}}{}^{\frac{3}{p}-2}_{p,1})}}
    \\&\leq C\|m\otimes m\|_{{L^1(0,T;\dot{B}^{-\frac{1}{2}}_{2,1})}}\\
	&\leq C \|m\|_{{L^r(0,T;\dot{B}^{\frac{3}{q}-3+\frac{4}{r}}_{q,1})}}\|m\|_{{L^{r^\prime}(0,T;\dot{B}^{\frac{3}{q}-3+\frac{4}{r^\prime}}_{q,1})}},\\
		\|\div(m\otimes m)\|_{{L^1(0,T;\widehat{\dot{B}}{}^{\frac{3}{p}-1}_{p,1})}}&\leq C\|m\otimes m\|_{{L^1(0,T;\widehat{\dot{B}}{}^{\frac{3}{p}}_{p,1})}}
        \\&\leq C\|m\otimes m\|_{{L^1(0,T;\dot{B}^{\frac{3}{2}}_{2,1})}} \label{nones1}\\
	&\leq C \|m\|_{{L^r(0,T;\dot{B}^{\frac{3}{q}-1+\frac{2}{r}}_{q,1})}}\|m\|_{{L^{r^\prime}(0,T;\dot{B}^{\frac{3}{q}-1+\frac{2}{r^\prime}}_{q,1})}}.
\end{align}
Based on \eqref{nones1}, using \lemref{keyproduct1} and \lemref{lemm-comp} in addition, we obtain\begin{align}
	\|&\div(I(\varepsilon a)m\otimes m)\|_{{L^1(0,T;\widehat{\dot{B}}{}^{\frac{3}{p}-3}_{p,1})}}\\
	&\leq C\|I(\varepsilon a)m\otimes m\|_{{L^1(0,T;\widehat{\dot{B}}{}^{\frac{3}{p}-2}_{p,1})}}
	\\&\leq C\|I(\varepsilon a)\|_{{L^\infty(0,T;\widehat{\dot{B}}{}^{\frac{3}{p}}_{p,1})}}\|m\otimes m\|_{{L^1(0,T;\widehat{\dot{B}}{}^{\frac{3}{p}-2}_{p,1})}}
\\
	&\leq C\|\varepsilon\nabla a\|_{{L^\infty(0,T;\widehat{\dot{B}}{}^{\frac{3}{p}-1}_{p,1})}} \|m\|_{{L^r(0,T;\dot{B}^{\frac{3}{q}-3+\frac{4}{r}}_{q,1})}}\|m\|_{{L^{r^\prime}(0,T;\dot{B}^{\frac{3}{q}-3+\frac{4}{r^\prime}}_{q,1})}},\\
		\|&\div(I(\varepsilon a)m\otimes m)\|_{{L^1(0,T;\widehat{\dot{B}}{}^{\frac{3}{p}-1}_{p,1})}} \label{nones2}\\
	&\leq C\|I(\varepsilon a)m\otimes m\|_{{L^1(0,T;\widehat{\dot{B}}{}^{\frac{3}{p}}_{p,1})}}
	\\&\leq C\|I(\varepsilon a)\|_{{L^\infty(0,T;\widehat{\dot{B}}{}^{\frac{3}{p}}_{p,1})}}\|m\otimes m\|_{{L^1(0,T;\widehat{\dot{B}}{}^{\frac{3}{p}}_{p,1})}}
	\\
	&\leq C\|\varepsilon\nabla a\|_{{L^\infty(0,T;\widehat{\dot{B}}{}^{\frac{3}{p}-1}_{p,1})}} \|m\|_{{L^r(0,T;\dot{B}^{\frac{3}{q}-1+\frac{2}{r}}_{q,1})}}\|m\|_{{L^{r^\prime}(0,T;\dot{B}^{\frac{3}{q}-1+\frac{2}{r^\prime}}_{q,1})}}.
\end{align}
Observe that\begin{equation}
	\begin{split}
		I(b)=b-bI(b).
	\end{split}
\end{equation}
Hence, making use of \lemref{keyproduct1}, \lemref{lemm-comp}, \eqref{needembedding} and \lemref{keyproduct2} (with $(s,s_1,s_3)=(1,2-\frac{4}{r},3-\frac{4}{r})$), we have
\begin{align}
\|&\Delta(I(\varepsilon a)m)\|_{{L^1(0,T;\widehat{\dot{B}}{}^{\frac{3}{p}-3}_{p,1})}}+\|\nabla\div(I(\varepsilon a)m)\|_{{L^1(0,T;\widehat{\dot{B}}{}^{\frac{3}{p}-3}_{p,1})}}\\
&\leq C\|I(\varepsilon a)m\|_{{L^1(0,T;\widehat{\dot{B}}{}^{\frac{3}{p}-1}_{p,1})}}\\
&\leq C\|\varepsilon am\|_{{L^1(0,T;\widehat{\dot{B}}{}^{\frac{3}{p}-1}_{p,1})}}+C\|\varepsilon aI(\varepsilon a)m\|_{{L^1(0,T;\widehat{\dot{B}}{}^{\frac{3}{p}-1}_{p,1})}}\\
&\leq C\|\varepsilon am\|_{{L^1(0,T;\widehat{\dot{B}}{}^{\frac{3}{p}-1}_{p,1})}}+C\|I(\varepsilon a)\|_{{L^\infty(0,T;\widehat{\dot{B}}{}^{\frac{3}{p}}_{p,1})}}\|\varepsilon am\|_{{L^1(0,T;\widehat{\dot{B}}{}^{\frac{3}{p}-1}_{p,1})}}\label{nones3}\\
&\leq C\left(1+\|\varepsilon\nabla a\|_{{L^\infty(0,T;\widehat{\dot{B}}{}^{\frac{3}{p}-1}_{p,1})}}\right)\|\varepsilon am\|_{{L^1(0,T;{\dot{B}}{}^{\frac{1}{2}}_{2,1})}} \\
&\leq C\left(1+\|\varepsilon\nabla a\|_{{L^\infty(0,T;\widehat{\dot{B}}{}^{\frac{3}{p}-1}_{p,1})}}\right) \left(\|\varepsilon\nabla a\|_{{L^r(0,T;\dot{B}^{\frac{3}{q}-3+\frac{4}{r}}_{q,1})}}\|m\|_{{L^{r^\prime}(0,T;\dot{B}^{\frac{3}{q}-3+\frac{4}{r^\prime}}_{q,1})}}\right.\\
&\qquad\qquad\left.+\|m\|_{{L^r(0,T;\dot{B}^{\frac{3}{q}-3+\frac{4}{r}}_{q,1})}}\|\varepsilon\nabla a\|_{{L^{r^\prime}(0,T;\dot{B}^{\frac{3}{q}-3+\frac{4}{r^\prime}}_{q,1})}}\right).
\end{align}
Thanks to \lemref{keyproduct3} (with $(s,s_1,s_2)=(1,0,1-\frac{2}{r})$) and \lemref{lemm-comp}, we also have\begin{align}
\|&\Delta(I(\varepsilon a)m)\|_{{L^1(0,T;\widehat{\dot{B}}{}^{\frac{3}{p}-1}_{p,1})}}+\|\nabla\div(I(\varepsilon a)m)\|_{{L^1(0,T;\widehat{\dot{B}}{}^{\frac{3}{p}-1}_{p,1})}}\\
&\leq C\|I(\varepsilon a)m\|_{{L^1(0,T;\widehat{\dot{B}}{}^{\frac{3}{p}+1}_{p,1})}}\\
&\leq C\|I(\varepsilon a)\|_{{L^\infty(0,T;\widehat{\dot{B}}{}^{\frac{3}{p}}_{p,1})}}\|m\|_{{L^1(0,T;\widehat{\dot{B}}{}^{\frac{3}{p}+1}_{p,1})}} \label{nones4}\\
&\qquad+C\|m\|_{{L^r(0,T;\dot{B}^{\frac{3}{q}-1+\frac{2}{r}}_{q,1})}}\|I(\varepsilon a)\|_{{L^{r^\prime}(0,T;\dot{B}^{\frac{3}{q}+\frac{2}{r^\prime}}_{q,1})}}\\
&\leq C\|\varepsilon\nabla  a\|_{{L^\infty(0,T;\widehat{\dot{B}}{}^{\frac{3}{p}-1}_{p,1})}}\|m\|_{{L^1(0,T;\widehat{\dot{B}}{}^{\frac{3}{p}+1}_{p,1})}}\\
&\qquad+C\|m\|_{{L^r(0,T;\dot{B}^{\frac{3}{q}-1+\frac{2}{r}}_{q,1})}}\|\varepsilon\nabla a\|_{{L^{r^\prime}(0,T;\dot{B}^{\frac{3}{q}-1+\frac{2}{r^\prime}}_{q,1})}}.
\end{align}
Let $G(b):=\int_{0}^{b}J(s)ds$, then we see that\begin{equation}\label{thegfunction}
	\begin{split}
	\frac{1}{\varepsilon}J(\varepsilon a)\nabla a=\frac{1}{\varepsilon^2}\nabla(G(\varepsilon a)).
	\end{split}
\end{equation}
Notice that $G(0)=G^\prime(0)=0$, so there exists a smooth function $H$ satisfying $H(0)=0$ such that\begin{equation}\label{thegfunction1}
	\begin{split}
		G(b)=\frac{G^{\prime\prime}(0)}{2}b^2+H(b)b^2.
	\end{split}
\end{equation}
Hence, resorting to \lemref{keyproduct1}, \lemref{lemm-comp}, \eqref{needembedding}, \lemref{keyproduct2} once again, we have\begin{align}
	\frac{1}{\varepsilon}&\|{J}(\varepsilon a)\nabla a\|_{L^1(0,T;\widehat{\dot{B}}{}^{\frac{3}{p}-3}_{p,1})}\\&\leq 	\frac{C}{\varepsilon^2}\|G(\varepsilon a)\|_{L^1(0,T;\widehat{\dot{B}}{}^{\frac{3}{p}-2}_{p,1})}\\
	&\leq C \left(1+\|H(\varepsilon a)\|_{L^\infty(0,T;\widehat{\dot{B}}{}^{\frac{3}{p}}_{p,1})}\right)\|a^2\|_{L^1(0,T;\widehat{\dot{B}}{}^{\frac{3}{p}-2}_{p,1})}\\
	&\leq C \left(1+\|\varepsilon\nabla a\|_{L^\infty(0,T;\widehat{\dot{B}}{}^{\frac{3}{p}-1}_{p,1})}\right)\|a^2\|_{L^1(0,T;\widehat{\dot{B}}{}^{-\frac{1}{2}}_{2,1})}\\
	&\leq C \left(1+\|\varepsilon\nabla a\|_{L^\infty(0,T;\widehat{\dot{B}}{}^{\frac{3}{p}-1}_{p,1})}\right)\|a\|_{{L^r(0,T;\dot{B}^{\frac{3}{q}-3+\frac{4}{r}}_{q,1})}}\|a\|_{{L^{r^\prime}(0,T;\dot{B}^{\frac{3}{q}-3+\frac{4}{r^\prime}}_{q,1})}},\\
	\frac{1}{\varepsilon}&\|{J}(\varepsilon a)\nabla a\|_{L^1(0,T;\widehat{\dot{B}}{}^{\frac{3}{p}-1}_{p,1})}\label{nones5}\\&\leq 	\frac{C}{\varepsilon^2}\|G(\varepsilon a)\|_{L^1(0,T;\widehat{\dot{B}}{}^{\frac{3}{p}}_{p,1})}\\
	&\leq C \left(1+\|H(\varepsilon a)\|_{L^\infty(0,T;\widehat{\dot{B}}{}^{\frac{3}{p}}_{p,1})}\right)\|a^2\|_{L^1(0,T;\widehat{\dot{B}}{}^{\frac{3}{p}}_{p,1})}\\
	&\leq C \left(1+\|\varepsilon\nabla a\|_{L^\infty(0,T;\widehat{\dot{B}}{}^{\frac{3}{p}-1}_{p,1})}\right)\|a^2\|_{L^1(0,T;\widehat{\dot{B}}{}^{\frac{3}{2}}_{2,1})}\\
	&\leq C \left(1+\|\varepsilon\nabla a\|_{L^\infty(0,T;\widehat{\dot{B}}{}^{\frac{3}{p}-1}_{p,1})}\right)\|a\|_{{L^r(0,T;\dot{B}^{\frac{3}{q}-1+\frac{2}{r}}_{q,1})}}\|a\|_{{L^{r^\prime}(0,T;\dot{B}^{\frac{3}{q}-1+\frac{2}{r^\prime}}_{q,1})}}.
\end{align}
As for the remaining terms, we can use \lemref{keyproduct1} to bound them as\begin{align}
	\|\varepsilon^2\nabla(a\Delta a)\|_{L^1(0,T;\widehat{\dot{B}}{}^{\frac{3}{p}-3}_{p,1})}&\leq C	\|\varepsilon^2a\Delta a\|_{L^1(0,T;\widehat{\dot{B}}{}^{\frac{3}{p}-2}_{p,1})}\\
	&\leq C	\|\varepsilon\nabla  a\|_{L^\infty(0,T;\widehat{\dot{B}}{}^{\frac{3}{p}-1}_{p,1})}\|\varepsilon\nabla  a\|_{L^1(0,T;\widehat{\dot{B}}{}^{\frac{3}{p}-1}_{p,1})},\\
	\|\varepsilon^2\nabla(a\Delta a)\|_{L^1(0,T;\widehat{\dot{B}}{}^{\frac{3}{p}-1}_{p,1})}&\leq C	\|\varepsilon^2a\Delta a\|_{L^1(0,T;\widehat{\dot{B}}{}^{\frac{3}{p}}_{p,1})}\\
	&\leq C	\|\varepsilon\nabla  a\|_{L^\infty(0,T;\widehat{\dot{B}}{}^{\frac{3}{p}-1}_{p,1})}\|\varepsilon\nabla  a\|_{L^1(0,T;\widehat{\dot{B}}{}^{\frac{3}{p}+1}_{p,1})},\\
		\|\varepsilon^2\nabla(|\nabla a|^2)\|_{L^1(0,T;\widehat{\dot{B}}{}^{\frac{3}{p}-3}_{p,1})}&+\|\varepsilon^2\div(\nabla a\otimes \nabla a)\|_{L^1(0,T;\widehat{\dot{B}}{}^{\frac{3}{p}-3}_{p,1})}\\
		&\leq C\varepsilon^2\|\nabla a\otimes \nabla a\|_{L^1(0,T;\widehat{\dot{B}}{}^{\frac{3}{p}-2}_{p,1})} \label{nones6}\\
		&\leq C	\|\varepsilon\nabla  a\|_{L^\infty(0,T;\widehat{\dot{B}}{}^{\frac{3}{p}-1}_{p,1})}\|\varepsilon\nabla  a\|_{L^1(0,T;\widehat{\dot{B}}{}^{\frac{3}{p}-1}_{p,1})},\\
		\|\varepsilon^2\nabla(|\nabla a|^2)\|_{L^1(0,T;\widehat{\dot{B}}{}^{\frac{3}{p}-1}_{p,1})}&+\|\varepsilon^2\div(\nabla a\otimes \nabla a)\|_{L^1(0,T;\widehat{\dot{B}}{}^{\frac{3}{p}-1}_{p,1})}\\
		&\leq C\varepsilon^2\|\nabla a\otimes \nabla a\|_{L^1(0,T;\widehat{\dot{B}}{}^{\frac{3}{p}}_{p,1})}\\
		&\leq C	\|\varepsilon\nabla  a\|_{L^\infty(0,T;\widehat{\dot{B}}{}^{\frac{3}{p}-1}_{p,1})}\|\varepsilon\nabla  a\|_{L^1(0,T;\widehat{\dot{B}}{}^{\frac{3}{p}+1}_{p,1})}.
\end{align}
Moreover, we see that\begin{equation}
	\begin{split}
		\|(a,\varepsilon\nabla a,m)\|_{{L^{r^\prime}(0,T;\dot{B}^{\frac{3}{q}-3+\frac{4}{r^\prime}}_{q,1})}}&\leq C\|(a,\varepsilon\nabla a,m)\|_{{L^{r^\prime}(0,T;\widehat{\dot{B}}{}^{\frac{3}{p}-3+\frac{4}{r^\prime}}_{p,1})}}\\
		&\leq C\|(a,\varepsilon\nabla a,m)\|_{{L^{\infty}(0,T;\widehat{\dot{B}}{}^{\frac{3}{p}-3}_{p,1})}\cap {L^1(0,T;\widehat{\dot{B}}{}^{\frac{3}{p}+1}_{p,1})}},\\
		\|(a,\varepsilon\nabla a,m)\|_{{L^{r^\prime}(0,T;\dot{B}^{\frac{3}{q}-1+\frac{2}{r^\prime}}_{q,1})}}&\leq C\|(a,\varepsilon\nabla a,m)\|_{{L^{r^\prime}(0,T;\widehat{\dot{B}}{}^{\frac{3}{p}-1+\frac{2}{r^\prime}}_{p,1})}}\label{nones7}\\
		&\leq C\|(a,\varepsilon\nabla a,m)\|_{{L^{\infty}(0,T;\widehat{\dot{B}}{}^{\frac{3}{p}-1}_{p,1})}\cap {L^1(0,T;\widehat{\dot{B}}{}^{\frac{3}{p}+1}_{p,1})}},
	\end{split}
\end{equation}which stems from the embedding relation:\begin{align}
\widehat{\dot{B}}{}^{\frac{3}{p}+s}_{p,1}(\mathbb{R}^3)\hookrightarrow{\dot{B}}{}^{\frac{3}{q}+s}_{q,1}(\mathbb{R}^3)\quad \text{for all } s\in\mathbb{R},\label{needembedding2}
\end{align}and the interpolation inequalities:\begin{equation}\label{needinter2}
\begin{split}
	\|f\|_{L^{r^\prime}(0,T;\widehat{\dot{B}}{}^{\frac{3}{p}-3+\frac{4}{r^\prime}}_{p,1})}&\leq \|f\|_{L^{\infty}(0,T;\widehat{\dot{B}}{}^{\frac{3}{p}-3}_{p,1})}^{\frac{1}{r}}\|f\|_{L^{1}(0,T;\widehat{\dot{B}}{}^{\frac{3}{p}+1}_{p,1})}^{1-\frac{1}{r}},\\
	\|f\|_{L^{r^\prime}(0,T;\widehat{\dot{B}}{}^{\frac{3}{p}-1+\frac{2}{r^\prime}}_{p,1})}&\leq \|f\|_{L^{\infty}(0,T;\widehat{\dot{B}}{}^{\frac{3}{p}-1}_{p,1})}^{\frac{1}{r}}\|f\|_{L^{1}(0,T;\widehat{\dot{B}}{}^{\frac{3}{p}+1}_{p,1})}^{1-\frac{1}{r}}.
\end{split}
\end{equation}
Thus, substituting \eqref{nones1}, \eqref{nones2}, \eqref{nones3}, \eqref{nones4}, \eqref{nones5} and \eqref{nones6} into \eqref{zhuyaoes1}, and using \eqref{nones7}, we finally obtain desired estimates \eqref{firstapriories}. This completes the proof.
\end{proof}

\begin{lemm}\label{xygj2}
Let $2\leq p<q<3$ and $2<r<\infty$ satisfy\begin{equation}\label{qrcondition}
	\begin{split}
	\frac{1}{r}\leq\min\left\{\frac{1}{p}-\frac{1}{q},\frac{3}{2q}-\frac{1}{4}\right\}.
	\end{split}
\end{equation}
Then, there exists a constant $C_2=C_2(p,q,r,\mu,\kappa,P) \geq 1$ such that for any $\beta>16$, $\Omega\in\mathbb{R}\setminus\{0\}$ and $0 < \varepsilon \leq \varepsilon_0$ with $|\Omega|\varepsilon\leq 1$, the solution $(a,m)$ to \eqref{eq:CNSKC-2} enjoys 
\begin{align}
	\mathcal{D}_{p,q,r}(T)&\leq C_2\left(|\Omega|^{\frac{2}{r}}\varepsilon^{\frac{2}{r}}+\varepsilon\beta+|\Omega|^{-\frac{1}{r}}\beta^{2-\frac{2}{r}}(1+\varepsilon\beta)\right)\\
	&\qquad\times \Big(\mathcal{E}_{p,0}+(1+\mathcal{D}_{p,q,r}(T))\mathcal{D}_{p,q,r}(T)\mathcal{E}_p(T)\Big)\\
	&\quad+C_2\left(1+\mathcal{D}_{p,q,r}(T)\right)\|(a_0,\varepsilon \nabla a_0,m_0)\|_{\widehat{\dot{B}}{}^{\frac{3}{p}-1}_{p,1}}^{h;\frac{\beta}{16}}\label{secondapriories}\\
	&\quad+C_2\left(1+\mathcal{D}_{p,q,r}(T)\right)\left[\mathcal{D}_{p,q,r}(T)\right]^{\frac{r}{r-1}}\left[\mathcal{E}_{p}(T)\right]^{\frac{r-2}{r-1}}\\
	&\quad+C_2\left(1+\mathcal{D}_{p,q,r}(T)\right)\left[\mathcal{D}_{p,q,r}(T)\right]^2\mathcal{E}_{p}(T)
\end{align}
for all $T \in (0,T_{\rm max}]$.
\end{lemm}
\begin{proof}
According to embeddings, Berstein inequalities and frequency cut-off, we see that\begin{align}
	\|(a,\varepsilon\nabla a,m)\|_{{\widetilde{L^r}(0,T;\dot{B}^{\frac{3}{q}-3+\frac{4}{r}}_{q,1})}}^{\ell;|\Omega|\varepsilon}&\leq C\|(a,\varepsilon\nabla a,m)\|_{{\widetilde{L^r}(0,T;\widehat{\dot{B}}{}^{\frac{3}{p}-3+\frac{4}{r}}_{p,1})}}^{\ell;|\Omega|\varepsilon},\\
\|\varepsilon\nabla a\|_{{\widetilde{L^r}(0,T;\dot{B}^{\frac{3}{q}-3+\frac{4}{r}}_{q,1})}}^{m;|\Omega|\varepsilon,\beta}&\leq \varepsilon\beta\| a\|_{{\widetilde{L^r}(0,T;\dot{B}^{\frac{3}{q}-3+\frac{4}{r}}_{q,1})}}^{m;|\Omega|\varepsilon,\beta}\\
\|(a,\varepsilon\nabla a,m)\|_{{\widetilde{L^r}(0,T;\dot{B}^{\frac{3}{q}-3+\frac{4}{r}}_{q,1})}}^{h;\beta}&\leq\beta^{-2+\frac{2}{r}}\|(a,\varepsilon\nabla a,m)\|_{{\widetilde{L^r}(0,T;\dot{B}^{\frac{3}{q}-1+\frac{2}{r}}_{q,1})}}^{h;\beta}\\
&\leq C\beta^{-2+\frac{2}{r}}\|(a,\varepsilon\nabla a,m)\|_{{\widetilde{L^r}(0,T;\widehat{\dot{B}}{}^{\frac{3}{p}-1+\frac{2}{r}}_{p,1})}}^{h;\beta}.
\end{align}
Then, we have
\begin{align}
    \|(a,\varepsilon\nabla a,m)\|_{{\widetilde{L^r}(0,T;\dot{B}^{\frac{3}{q}-3+\frac{4}{r}}_{q,1})}}
    &=  
    \|(a,\varepsilon\nabla a,m)\|_{{\widetilde{L^r}(0,T;\dot{B}^{\frac{3}{q}-3+\frac{4}{r}}_{q,1})}}^{\ell;|\Omega|\varepsilon}\\
    &\quad +\|(a,\varepsilon\nabla a,m)\|_{{\widetilde{L^r}(0,T;\dot{B}^{\frac{3}{q}-3+\frac{4}{r}}_{q,1})}}^{m;|\Omega|\varepsilon,\beta}\\
    &\quad +\|(a,\varepsilon\nabla a,m)\|_{{\widetilde{L^r}(0,T;\dot{B}^{\frac{3}{q}-3+\frac{4}{r}}_{q,1})}}^{h;\beta}\label{fdgj1}
    \\
    &\leq C\|(a,\varepsilon\nabla a,m)\|_{{\widetilde{L^r}(0,T;\widehat{\dot{B}}{}^{\frac{3}{p}-3+\frac{4}{r}}_{p,1})}}^{\ell;|\Omega|\varepsilon}\\
    &\quad+C(1+\varepsilon\beta)\|(a,m)\|_{{\widetilde{L^r}(0,T;\dot{B}^{\frac{3}{q}-3+\frac{4}{r}}_{q,1})}}^{m;|\Omega|\varepsilon,\beta}\\
    &\quad+C\beta^{-2+\frac{2}{r}}\|(a,\varepsilon\nabla a,m)\|_{{\widetilde{L^r}(0,T;\widehat{\dot{B}}{}^{\frac{3}{p}-1+\frac{2}{r}}_{p,1})}}^{h;\beta}.
\end{align}
Similarly, we also have\begin{align}
	\|(a,\varepsilon\nabla a,m)\|_{{\widetilde{L^r}(0,T;\dot{B}^{\frac{3}{q}-1+\frac{2}{r}}_{q,1})}}&\leq C\left(|\Omega|\varepsilon\right)^{2-\frac{2}{r}}\|(a,\varepsilon\nabla a,m)\|_{{\widetilde{L^r}(0,T;\widehat{\dot{B}}{}^{\frac{3}{p}-3+\frac{4}{r}}_{p,1})}}^{\ell;|\Omega|\varepsilon}\\
	&\quad+C\beta^{2-\frac{2}{r}}(1+\varepsilon\beta)\|(a,m)\|_{{\widetilde{L^r}(0,T;\dot{B}^{\frac{3}{q}-3+\frac{4}{r}}_{q,1})}}^{m;|\Omega|\varepsilon,\beta}\\
	&\quad+C\|(a,\varepsilon\nabla a,m)\|_{{\widetilde{L^r}(0,T;\widehat{\dot{B}}^{\frac{3}{p}-1+\frac{2}{r}}_{p,1})}}^{h;\beta},\label{fdgj2}\\
		\|\varepsilon\nabla a\|_{{\widetilde{L^\infty}(0,T;\widehat{\dot{B}}{}^{\frac{3}{p}-1}_{p,1})}}\leq C&|\Omega|\varepsilon^2	\| a\|_{{\widetilde{L^\infty}(0,T;\widehat{\dot{B}}{}^{\frac{3}{p}-1}_{p,1})}}^{\ell;|\Omega|\varepsilon}+C\varepsilon\beta	\| a\|_{{\widetilde{L^\infty}(0,T;\widehat{\dot{B}}{}^{\frac{3}{p}-1}_{p,1})}}^{m;|\Omega|\varepsilon,\beta}\\&\quad+C\|\varepsilon\nabla a\|_{{\widetilde{L^\infty}(0,T;\widehat{\dot{B}}{}^{\frac{3}{p}-1}_{p,1})}}^{h;\beta}.
\end{align}
Hence, combining \eqref{fdgj1} and \eqref{fdgj2} implies\begin{align}
    \mathcal{D}_{p,q,r}(T)&\leq C\|(a,\varepsilon\nabla a,m)\|_{{\widetilde{L^r}(0,T;\widehat{\dot{B}}{}^{\frac{3}{p}-3+\frac{4}{r}}_{p,1})}}^{\ell;|\Omega|\varepsilon}\\
    &\quad+C\beta^{2-\frac{2}{r}}(1+\varepsilon\beta)\|(a,m)\|_{{\widetilde{L^r}(0,T;\dot{B}^{\frac{3}{q}-3+\frac{4}{r}}_{q,1})}}^{m;|\Omega|\varepsilon,\beta}   \\
    &\quad+C\|(a,\varepsilon\nabla a,m)\|_{{\widetilde{L^r}(0,T;\widehat{\dot{B}}{}^{\frac{3}{p}-1+\frac{2}{r}}_{p,1})}}^{h;\beta}\label{qidaigood}\\
    &\quad+C\varepsilon\beta\left(\|a\|_{{\widetilde{L^\infty}(0,T;\widehat{\dot{B}}{}^{\frac{3}{p}-1}_{p,1})}}^{\ell;|\Omega|\varepsilon}+\| a\|_{{\widetilde{L^\infty}(0,T;\widehat{\dot{B}}{}^{\frac{3}{p}-1}_{p,1})}}^{m;|\Omega|\varepsilon,\beta}  \right)\\
    &\quad+C\|\varepsilon\nabla a\|_{{\widetilde{L^\infty}(0,T;\widehat{\dot{B}}{}^{\frac{3}{p}-1}_{p,1})}}^{h;\beta}.
\end{align}
On the other hand, applying \lemref{energylemma} gives us\begin{align}
	\|&(a,\varepsilon\nabla a,m)\|_{{\widetilde{L^r}(0,T;\widehat{\dot{B}}{}^{\frac{3}{p}-3+\frac{4}{r}}_{p,1})}}^{\ell;|\Omega|\varepsilon}\\&\leq C|\Omega|^{\frac{2}{r}}\varepsilon^{\frac{2}{r}}\|(a_0,\varepsilon \nabla a_0,m_0)\|_{\widehat{\dot{B}}{}^{\frac{3}{p}-3}_{p,1}}^{\ell;|\Omega|\varepsilon}+C|\Omega|^{\frac{2}{r}}\varepsilon^{\frac{2}{r}}\|N_\varepsilon[a,m]\|_{{L^{1}}(0,T;\widehat{\dot{B}}{}^{\frac{3}{p}-3}_{p,1})}^{\ell;|\Omega|\varepsilon},\\
	\| &a\|_{{\widetilde{L^\infty}(0,T;\widehat{\dot{B}}{}^{\frac{3}{p}-1}_{p,1})}}^{\ell;|\Omega|\varepsilon}\\&\leq C\|(a_0,\varepsilon \nabla a_0,m_0)\|_{\widehat{\dot{B}}{}^{\frac{3}{p}-1}_{p,1}}^{\ell;|\Omega|\varepsilon}+C\|N_\varepsilon[a,m]\|_{{L^{1}}(0,T;\widehat{\dot{B}}{}^{\frac{3}{p}-1}_{p,1})}^{\ell;|\Omega|\varepsilon},\\
    \| &a\|_{{\widetilde{L^\infty}(0,T;\widehat{\dot{B}}{}^{\frac{3}{p}-1}_{p,1})}}^{m;|\Omega|\varepsilon,\beta}  \label{fdgj3}
	\\&\leq C\|(a_0,\varepsilon \nabla a_0,m_0)\|_{\widehat{\dot{B}}{}^{\frac{3}{p}-1}_{p,1}}^{m;|\Omega|\varepsilon,\beta}+C\|N_\varepsilon[a,m]\|_{{L^{1}}(0,T;\widehat{\dot{B}}{}^{\frac{3}{p}-1}_{p,1})}^{m;|\Omega|\varepsilon,\beta},\\
	\|&(a,\varepsilon\nabla a,m)\|_{{\widetilde{L^r}(0,T;\widehat{\dot{B}}{}^{\frac{3}{p}-1+\frac{2}{r}}_{p,1})}}^{h;\beta}+\|\varepsilon\nabla a\|_{{\widetilde{L^\infty}(0,T;\widehat{\dot{B}}{}^{\frac{3}{p}-1}_{p,1})}}^{h;\beta}\\
    &\leq C\|(a_0,\varepsilon \nabla a_0,m_0)\|_{\widehat{\dot{B}}{}^{\frac{3}{p}-1}_{p,1}}^{h;\beta}+C\|N_\varepsilon[a,m]\|_{{L^{1}}(0,T;\widehat{\dot{B}}{}^{\frac{3}{p}-1}_{p,1})}^{h;\beta}.
\end{align}
Besides, resorting to \lemref{dispersivelemma}, we get\begin{equation} \label{fdgj4}
	\begin{split}
			\|(a,m)\|_{{\widetilde{L^r}(0,T;\dot{B}^{\frac{3}{q}-3+\frac{4}{r}}_{q,1})}}^{m;|\Omega|\varepsilon,\beta}&\leq C|\Omega|^{-\frac{1}{r}}\left(\|(a_0,\varepsilon \nabla a_0,m_0)\|_{\widehat{\dot{B}}{}^{\frac{3}{p}-3+\frac{4}{r}}_{p,1}}\right.\\
		&\qquad\left.+\|N_\varepsilon[a,m]\|_{{L^{1}}(0,T;\widehat{\dot{B}}{}^{\frac{3}{p}-3+\frac{4}{r}}_{p,1})}\right)
            \\&\leq C|\Omega|^{-\frac{1}{r}}\left(\|(a_0,\varepsilon \nabla a_0,m_0)\|_{\widehat{\dot{B}}{}^{\frac{3}{p}-3}_{p,1}\cap\widehat{\dot{B}}{}^{\frac{3}{p}-1}_{p,1} }\right.\\
		&\qquad\left.+\|N_\varepsilon[a,m]\|_{{L^{1}}(0,T;\widehat{\dot{B}}{}^{\frac{3}{p}-3}_{p,1})\cap {L^{1}}(0,T;\widehat{\dot{B}}{}^{\frac{3}{p}-1}_{p,1})}\right),
	\end{split}
\end{equation}where we have used the interpolation inequality:\begin{align}
\|f\|_{\widehat{\dot{B}}{}^{\frac{3}{p}-3+\frac{4}{r}}_{p,1}}\leq \|f\|_{\widehat{\dot{B}}{}^{\frac{3}{p}-3}_{p,1}}^{1-\frac{2}{r}}\|f\|_{\widehat{\dot{B}}{}^{\frac{3}{p}-1}_{p,1}}^{\frac{2}{r}}. 
\end{align}
Thus, plugging \eqref{fdgj3} and \eqref{fdgj4} into \eqref{qidaigood}, we can conclude that\begin{align}
\mathcal{D}_{p,q,r}(T)&\leq C\left(|\Omega|^{\frac{2}{r}}\varepsilon^{\frac{2}{r}}+\varepsilon\beta+|\Omega|^{-\frac{1}{r}}\beta^{2-\frac{2}{r}}(1+\varepsilon\beta)\right)\\
&\qquad\times\left(\|(a_0,\varepsilon \nabla a_0,m_0)\|_{\widehat{\dot{B}}{}^{\frac{3}{p}-1}_{p,1}\cap \widehat{\dot{B}}{}^{\frac{3}{p}-3}_{p,1}}\right.\\
&\qquad\qquad\left.+\|N_\varepsilon[a,m]\|_{{L^{1}}(0,T;\widehat{\dot{B}}{}^{\frac{3}{p}-3}_{p,1})\cap{L^{1}}(0,T;\widehat{\dot{B}}{}^{\frac{3}{p}-1}_{p,1})}\right)\label{fqzgle}\\
&\quad+C\|(a_0,\varepsilon \nabla a_0,m_0)\|_{\widehat{\dot{B}}{}^{\frac{3}{p}-1}_{p,1}}^{h;\beta}+C\|N_\varepsilon[a,m]\|_{{L^{1}}(0,T;\widehat{\dot{B}}{}^{\frac{3}{p}-1}_{p,1})}^{h;\beta}.
\end{align}
From the proof of \lemref{xygj1}, we already established that
\begin{equation}\label{supernones}
	\begin{split}
		\|N_\varepsilon[a,m]\|_{{L^{1}}(0,T;\widehat{\dot{B}}{}^{\frac{3}{p}-3}_{p,1})\cap{L^{1}}(0,T;\widehat{\dot{B}}{}^{\frac{3}{p}-1}_{p,1})}\leq C\left(1+\mathcal{D}_{p,q,r}(T)\right)\mathcal{D}_{p,q,r}(T)\mathcal{E}_p(T).
	\end{split}
\end{equation}
Hence, it is left to bound $\|N_\varepsilon[a,m]\|_{{L^{1}}(0,T;\widehat{\dot{B}}{}^{\frac{3}{p}-1}_{p,1})}^{h;\beta}$. Reverting \eqref{nones1}, \eqref{nones2} and \eqref{nones5}, and using the following interpolation inequality:\begin{equation}\label{superinter}
	\begin{split}
		\|f\|_{L^{r^\prime}(0,T;{\dot{B}}{}^{\frac{3}{q}-1+\frac{2}{r^\prime}}_{q,1})}&\leq \|f\|_{L^{r}(0,T;{\dot{B}}{}^{\frac{3}{q}-1+\frac{2}{r}}_{q,1})}^{\frac{1}{r-1}}\|f\|_{L^{1}(0,T;{\dot{B}}^{\frac{3}{q}+1}_{q,1})}^{\frac{r-2}{r-1}}\\&\leq C \|f\|_{L^{r}(0,T;{\dot{B}}{}^{\frac{3}{q}-1+\frac{2}{r}}_{q,1})}^{\frac{1}{r-1}}\|f\|_{L^{1}(0,T;\widehat{\dot{B}}{}^{\frac{3}{p}+1}_{p,1})}^{\frac{r-2}{r-1}},
	\end{split}
\end{equation}we have\begin{align}
&\|\div\left((I(\varepsilon a)-1)m\otimes m\right)\|_{{L^{1}}(0,T;\widehat{\dot{B}}{}^{\frac{3}{p}-1}_{p,1})}^{h;\beta}\\
&\quad\leq C\left(1+\|\varepsilon\nabla a\|_{{L^\infty(0,T;\widehat{\dot{B}}{}^{\frac{3}{p}-1}_{p,1})}}\right) \|m\|_{{L^r(0,T;\dot{B}^{\frac{3}{q}-1+\frac{2}{r}}_{q,1})}}\|m\|_{{L^{r^\prime}(0,T;\dot{B}^{\frac{3}{q}-1+\frac{2}{r^\prime}}_{q,1})}}\\
&\quad\leq C\left(1+\|\varepsilon\nabla a\|_{{L^\infty(0,T;\widehat{\dot{B}}{}^{\frac{3}{p}-1}_{p,1})}}\right) \|m\|_{{L^r(0,T;\dot{B}^{\frac{3}{q}-1+\frac{2}{r}}_{q,1})}}^{\frac{r}{r-1}}\|m\|_{{L^{1}(0,T;\widehat{\dot{B}}{}^{\frac{3}{p}+1}_{p,1})}}^{\frac{r-2}{r-1}},\\
&\frac{1}{\varepsilon}\|{J}(\varepsilon a)\nabla a\|_{L^1(0,T;\widehat{\dot{B}}{}^{\frac{3}{p}-1}_{p,1})}^{h;\beta}  \label{ydjz1}\\
&\quad\leq C \left(1+\|\varepsilon\nabla a\|_{L^\infty(0,T;\widehat{\dot{B}}{}^{\frac{3}{p}-1}_{p,1})}\right)\|a\|_{{L^r(0,T;\dot{B}^{\frac{3}{q}-1+\frac{2}{r}}_{q,1})}}\|a\|_{{L^{r^\prime}(0,T;\dot{B}^{\frac{3}{q}-1+\frac{2}{r^\prime}}_{q,1})}}\\
&\quad\leq C \left(1+\|\varepsilon\nabla a\|_{L^\infty(0,T;\widehat{\dot{B}}{}^{\frac{3}{p}-1}_{p,1})}\right)\|a\|_{{L^r(0,T;\dot{B}^{\frac{3}{q}-1+\frac{2}{r}}_{q,1})}}^{\frac{r}{r-1}}\|a\|_{{L^{1}(0,T;\widehat{\dot{B}}{}^{\frac{3}{p}+1}_{p,1})}}^{\frac{r-2}{r-1}}.
\end{align}
Similar to \eqref{nones4}, using \lemref{keyproduct3}, \lemref{lemm-comp} and \eqref{superinter} leads to\begin{align}
	\|&\Delta(I(\varepsilon a)m)\|_{{L^1(0,T;\widehat{\dot{B}}{}^{\frac{3}{p}-1}_{p,1})}}^{h;\beta}+\|\nabla\div(I(\varepsilon a)m)\|_{{L^1(0,T;\widehat{\dot{B}}{}^{\frac{3}{p}-1}_{p,1})}}^{h;\beta}
    \\ &\leq C\|I(\varepsilon a)m\|_{{L^1(0,T;\widehat{\dot{B}}{}^{\frac{3}{p}+1}_{p,1})}}^{h;\beta}\\
	&\leq C\|I(\varepsilon a)\|_{{L^\infty(0,T;\widehat{\dot{B}}{}^{\frac{3}{p}}_{p,1})}}\|m\|_{{L^1(0,T;\widehat{\dot{B}}{}^{\frac{3}{p}+1}_{p,1})}}^{h;\frac{\beta}{16}}\\
	&\qquad+C\|m\|_{{L^r(0,T;\dot{B}^{\frac{3}{q}-1+\frac{2}{r}}_{q,1})}}\|I(\varepsilon a)\|_{{L^{r^\prime}(0,T;\dot{B}^{\frac{3}{q}+\frac{2}{r^\prime}}_{q,1})}}\label{ydjz2}\\
	&\leq C\|\varepsilon\nabla  a\|_{{L^\infty(0,T;\widehat{\dot{B}}{}^{\frac{3}{p}-1}_{p,1})}}\|m\|_{{L^1(0,T;\widehat{\dot{B}}{}^{\frac{3}{p}+1}_{p,1})}}^{h;\frac{\beta}{16}}\\
	&\qquad+C\|m\|_{{L^r(0,T;\dot{B}^{\frac{3}{q}-1+\frac{2}{r}}_{q,1})}}\|\varepsilon\nabla a\|_{L^{r}(0,T;{\dot{B}}^{\frac{3}{q}-1+\frac{2}{r}}_{q,1})}^{\frac{1}{r-1}}\|\varepsilon\nabla a\|_{L^{1}(0,T;\widehat{\dot{B}}{}^{\frac{3}{p}+1}_{p,1})}^{\frac{r-2}{r-1}}.
\end{align}
Using \lemref{keyproduct3} and \eqref{superinter} once again, we also obtain\begin{align}
	\|&\varepsilon^2\nabla(a\Delta a)\|_{L^1(0,T;\widehat{\dot{B}}{}^{\frac{3}{p}-1}_{p,1})}^{h;\beta}+\|\varepsilon^2\nabla(|\nabla a|^2)\|_{L^1(0,T;\widehat{\dot{B}}{}^{\frac{3}{p}-1}_{p,1})}^{h;\beta}\\
	&+\|\varepsilon^2\div(\nabla a\otimes \nabla a)\|_{L^1(0,T;\widehat{\dot{B}}{}^{\frac{3}{p}-1}_{p,1})}^{h;\beta}\\
	&\leq 	C\varepsilon^2\|a\Delta a\|_{L^1(0,T;\widehat{\dot{B}}{}^{\frac{3}{p}}_{p,1})}^{h;\beta}+C\varepsilon^2\|\nabla a\otimes \nabla a\|_{L^1(0,T;\widehat{\dot{B}}{}^{\frac{3}{p}}_{p,1})}^{h;\beta}\\
	&\leq C\|\varepsilon\nabla  a\|_{{L^\infty(0,T;\widehat{\dot{B}}{}^{\frac{3}{p}-1}_{p,1})}}\|\varepsilon\nabla  a\|_{{L^1(0,T;\widehat{\dot{B}}{}^{\frac{3}{p}+1}_{p,1})}}^{h;\frac{\beta}{16}}\label{ydjz3}\\&\qquad+\|\varepsilon\nabla a\|_{{L^r(0,T;\dot{B}^{\frac{3}{q}-1+\frac{2}{r}}_{q,1})}}\|\varepsilon\nabla a\|_{{L^{r^\prime}(0,T;\dot{B}^{\frac{3}{q}-1+\frac{2}{r^\prime}}_{q,1})}}\\
	&\leq C\|\varepsilon\nabla  a\|_{{L^\infty(0,T;\widehat{\dot{B}}{}^{\frac{3}{p}-1}_{p,1})}}\|\varepsilon\nabla  a\|_{{L^1(0,T;\widehat{\dot{B}}{}^{\frac{3}{p}+1}_{p,1})}}^{h;\frac{\beta}{16}}\\&\qquad+\|\varepsilon\nabla a\|_{{L^r(0,T;\dot{B}^{\frac{3}{q}-1+\frac{2}{r}}_{q,1})}}^{\frac{r}{r-1}}\|\varepsilon\nabla a\|_{{L^{1}(0,T;\widehat{\dot{B}}{}^{\frac{3}{p}+1}_{p,1})}}^{\frac{r-2}{r-1}}.
\end{align}
Therefore, combining \eqref{ydjz1}, \eqref{ydjz2} and \eqref{ydjz3}, we infer that\begin{equation}\label{supernones2}
	\begin{split}
		\|N_\varepsilon[a,m]\|_{{L^{1}}(0,T;\widehat{\dot{B}}{}^{\frac{3}{p}-1}_{p,1})}^{h;\beta}&\leq C\left(1+\mathcal{D}_{p,q,r}(T)\right)\left[\mathcal{D}_{p,q,r}(T)\right]^{\frac{r}{r-1}}\left[\mathcal{E}_{p}(T)\right]^{\frac{r-2}{r-1}}\\
		&\qquad+C\mathcal{D}_{p,q,r}(T)\|(a,\varepsilon\nabla  a,m)\|_{{L^1(0,T;\widehat{\dot{B}}{}^{\frac{3}{p}+1}_{p,1})}}^{h;\frac{\beta}{16}}.
	\end{split}
\end{equation}
In addition, it follows from \eqref{energyes2} and \eqref{supernones} that\begin{align}
	\|&(a,\varepsilon\nabla  a,m)\|_{{L^1(0,T;\widehat{\dot{B}}{}^{\frac{3}{p}+1}_{p,1})}}^{h;\frac{\beta}{16}}\\&\leq C\|(a_0,\varepsilon \nabla a_0,m_0)\|_{\widehat{\dot{B}}{}^{\frac{3}{p}-1}_{p,1}}^{h;\frac{\beta}{16}}+C\|N_\varepsilon[a,m]\|_{L^1(0,T;\widehat{\dot{B}}{}^{\frac{3}{p}-1}_{p,1})}^{h;\frac{\beta}{16}}\label{supernones3}\\
	&\leq C\|(a_0,\varepsilon \nabla a_0,m_0)\|_{\widehat{\dot{B}}{}^{\frac{3}{p}-1}_{p,1}}^{h;\frac{\beta}{16}}+C\left(1+\mathcal{D}_{p,q,r}(T)\right)\mathcal{D}_{p,q,r}(T)\mathcal{E}_p(T).
\end{align}
Finally, inserting \eqref{supernones}, \eqref{supernones2} and \eqref{supernones3} into \eqref{fqzgle} yields the desired estimates \eqref{secondapriories}. This completes the proof.
\end{proof}
\noindent\textbf{Third step: Global existence}\par
Let $2\leq p<3$. Then, we fix a pair $(q,r)\in(p,3)\times (2,\infty)$ satisfying \eqref{qrcondition}. Moreover, we assume that $\Omega \in \mathbb{R} \setminus \{ 0\}$ and $\varepsilon>0$ fulfill\begin{equation}\label{assumption2}
	\begin{split}
		\Omega_0 \leq |\Omega| \leq c_0 \frac{1}{\varepsilon},
	\end{split}
\end{equation}
where $\Omega_0\geq 1$ and $c_0\leq 1$ are positive constants to be determined later.
Hence, due to \propref{localwp}, by taking $c_0\leq\varepsilon_0$, we can construct a unique local solution $(a,m)$ to \eqref{eq:CNSKC-2} in the class \eqref{class123} that satisfies \eqref{boundnessofa}. Denote by $T_{\rm max}$ the maximal existence time of this local solution. To prove \thmref{thm:large}, it suffices to prove that $T_{\rm max}=\infty$.\par 
Next, we define\begin{equation}\label{deft}
	\begin{split}
		T^{*}:=
		\sup
		\left\{
		T \in (0,T_{\rm max})\ ;\ \mathcal{E}_p(T)\leq 2C_1\mathcal{E}_{p,0},
\quad \mathcal{D}_{p,q,r}(T)\leq \delta
		\right\},
	\end{split}
\end{equation}
where the constant $C_1$ is given in \lemref{xygj1} and the small positive constant $\delta\leq 1$ will be determined later. Let $\beta>16$ be a constant to be determined later. Thanks to \lemref{xygj1}, \lemref{xygj2}, \eqref{assumption2} and \eqref{deft}, we get that for any $0<T<T^{*}$,\begin{align}
	\mathcal{E}_p(T)&\leq C_1\mathcal{E}_{p,0}+2C_1^2\left(1+\delta\right)\delta\mathcal{E}_{p,0},\\
	\mathcal{D}_{p,q,r}(T)&\leq C_2\left(c_0^{\frac{2}{r}}+ c_0\Omega_0^{-1}\beta+|\Omega_0|^{-\frac{1}{r}}\beta^{2-\frac{2}{r}}(1+c_0\Omega_0^{-1}\beta)\right)\\
	&\qquad\quad\times \left(\mathcal{E}_{p,0}+2C_1\left(1+\delta\right)\delta\mathcal{E}_{p,0}\right)\\
	&\qquad+C_2(1+\delta)\|(a_0,\varepsilon \nabla a_0,m_0)\|_{\widehat{\dot{B}}{}^{\frac{3}{p}-1}_{p,1}}^{h;\frac{\beta}{16}}\\
	&\qquad+2C_1C_2\left(1+\delta\right)\delta^{\frac{r}{r-1}}\mathcal{E}_{p,0}^{\frac{r-2}{r-1}}+2C_1C_2\left(1+\delta\right)\delta^2 \mathcal{E}_{p,0}.
\end{align}
Now, we first choose $\delta$ sufficiently small to satisfy\begin{align}
	&C_1\left(1+\delta\right)\delta\leq\frac{1}{4},\\
	&2C_1C_2\left(1+\delta\right)\delta^{\frac{1}{r-1}}\mathcal{E}_{p,0}^{\frac{r-2}{r-1}}\leq\frac{1}{8},\\
	&2C_1C_2\left(1+\delta\right)\delta \mathcal{E}_{p,0}\leq\frac{1}{8}.
\end{align}
Then, we set $\beta$ sufficiently large such that\begin{equation}
	\begin{split}
		2C_2\|(a_0,\varepsilon \nabla a_0,m_0)\|_{\widehat{\dot{B}}{}^{\frac{3}{p}-1}_{p,1}}^{h;\frac{\beta}{16}}\leq\frac{\delta}{8}.
	\end{split}
\end{equation}
Furthermore, we take $\Omega_0$ sufficiently large and $c_0$ sufficiently small so that\begin{equation}
	\begin{split}
		5C_1C_2\mathcal{E}_{p,0}\left(c_0^{\frac{2}{r}}+ c_0\Omega_0^{-1}\beta+|\Omega_0|^{-\frac{1}{r}}\beta^{2-\frac{2}{r}}(1+c_0\Omega_0^{-1}\beta)\right)\leq\frac{\delta}{8}.
	\end{split}
\end{equation}
Thus, at this stage, we arrive at that for any $0<T<T^{*}$,\begin{equation}
	\begin{split}
		\mathcal{E}_p(T)\leq \frac{3C_1}{2}\mathcal{E}_{p,0},\quad \mathcal{D}_{p,q,r}(T)\leq \frac{\delta}{2}.
	\end{split}
\end{equation}Based on the estimates above and \eqref{deft}, by using a standard continuity argument, we can conclude that $T_{\rm max} \geq T^*=\infty$. Hence, the proof of \thmref{thm:large} is completed.

\appendix

\section{Proof of \propref{localwp}}\label{sec:A}
We present the proof of \propref{localwp}.
\begin{proof}[Proof of \propref{localwp}]
Let $T>0$. By the Duhamel principle, we rewrite the system \eqref{eq:CNSKC-2} as $(a,m)=\Phi_{\varepsilon}[a,m]$. Here, the map $\Phi_{\varepsilon}$ is defined by\begin{equation}\begin{split}
		\Phi_{\varepsilon}[a,m](t)
        :=
        \mathcal{G}_{\varepsilon}(t)\binom{a_0}{m_0}+\int_{0}^{t}\mathcal{G}_{\varepsilon}(t-\tau)\binom{0}{N_\varepsilon[a,m]}(\tau)d\tau,
	\end{split}
\end{equation}where $\{\mathcal{G}_{\varepsilon}(t)\}_{t>0}$ is the semigroup associated with the linear system\begin{equation}
\begin{cases}
	\partial_t a+\dfrac{1}{\varepsilon} \div m=0,
	\\
		\partial_t m-\mu\Delta m-(\mu+\lambda)\nabla\div m+\Omega ( e_3 \times m )+\dfrac{1}{\varepsilon} \nabla a-\kappa\varepsilon\nabla\Delta a=0.
\end{cases}
\end{equation}Moreover, we define the function space $Z^\varepsilon_T$ by\begin{align}
&Z^\varepsilon_T:=\left\{(a,m)\ ;\ 
\begin{aligned}
(a,\varepsilon\nabla a,m)&\in \tC ( [0,T] ; \fB_{p,1}^{\frac{3}{p}-3} ( \mathbb{R}^3 )\cap \fB_{p,1}^{\frac{3}{p}-1} ( \mathbb{R}^3 ))
 \\&\qquad\cap 
L^1( 0,T; \fB_{p,1}^{\frac{3}{p}+1} (\mathbb{R}^3)) \\
\|(a,m)\|_{Z^\varepsilon_T}&\leq \frac{1}{2C_0}
\end{aligned}\right\},\end{align}where
\begin{align}
&\|(a,m)\|_{Z^\varepsilon_T}:=\| a\|_{L^2(0,T;\widehat{\dot{B}}{}^{\frac{3}{p}-1}_{p,1})\cap L^2(0,T;\widehat{\dot{B}}{}^{\frac{3}{p}}_{p,1})}\\
&\qquad\qquad\qquad+\|\varepsilon\nabla a\|_{{\widetilde{L^\infty}(0,T;\widehat{\dot{B}}{}^{\frac{3}{p}-1}_{p,1})}\cap {L^2(0,T;\widehat{\dot{B}}{}^{\frac{3}{p}-1}_{p,1})}\cap {L^2(0,T;\widehat{\dot{B}}{}^{\frac{3}{p}}_{p,1})}\cap {L^1(0,T;\widehat{\dot{B}}{}^{\frac{3}{p}+1}_{p,1})}}\\
&\qquad\qquad\qquad+\|m\|_{{L^2(0,T;\widehat{\dot{B}}{}^{\frac{3}{p}-1}_{p,1})}\cap {L^2(0,T;\widehat{\dot{B}}{}^{\frac{3}{p}}_{p,1})}\cap {L^1(0,T;\widehat{\dot{B}}{}^{\frac{3}{p}+1}_{p,1})}},
\end{align}and $C_0\geq 1$ is a constant to be determined later. Similarly, we also adopt the following notations:\begin{align}
    &\|(a,m)\|_{Z^\varepsilon_T}^{\ell;|\Omega|\varepsilon}:=\| a\|_{L^2(0,T;\widehat{\dot{B}}{}^{\frac{3}{p}-1}_{p,1})\cap L^2(0,T;\widehat{\dot{B}}{}^{\frac{3}{p}}_{p,1})}^{\ell;|\Omega|\varepsilon}\\
&\qquad\qquad\qquad\qquad+\|\varepsilon\nabla a\|_{{\widetilde{L^\infty}(0,T;\widehat{\dot{B}}{}^{\frac{3}{p}-1}_{p,1})}\cap {L^2(0,T;\widehat{\dot{B}}{}^{\frac{3}{p}-1}_{p,1})}\cap {L^2(0,T;\widehat{\dot{B}}{}^{\frac{3}{p}}_{p,1})}\cap {L^1(0,T;\widehat{\dot{B}}{}^{\frac{3}{p}+1}_{p,1})}}^{\ell;|\Omega|\varepsilon}\\
&\qquad\qquad\qquad\qquad+\|m\|_{{L^2(0,T;\widehat{\dot{B}}{}^{\frac{3}{p}-1}_{p,1})}\cap {L^2(0,T;\widehat{\dot{B}}{}^{\frac{3}{p}}_{p,1})}\cap {L^1(0,T;\widehat{\dot{B}}{}^{\frac{3}{p}+1}_{p,1})}}^{\ell;|\Omega|\varepsilon},\\
&\|(a,m)\|_{Z^\varepsilon_T}^{h;|\Omega|\varepsilon}:=\| a\|_{L^2(0,T;\widehat{\dot{B}}{}^{\frac{3}{p}-1}_{p,1})\cap L^2(0,T;\widehat{\dot{B}}{}^{\frac{3}{p}}_{p,1})}^{h;|\Omega|\varepsilon}\\
&\qquad\qquad\qquad\qquad+\|\varepsilon\nabla a\|_{{\widetilde{L^\infty}(0,T;\widehat{\dot{B}}{}^{\frac{3}{p}-1}_{p,1})}\cap {L^2(0,T;\widehat{\dot{B}}{}^{\frac{3}{p}-1}_{p,1})}\cap {L^2(0,T;\widehat{\dot{B}}{}^{\frac{3}{p}}_{p,1})}\cap {L^1(0,T;\widehat{\dot{B}}{}^{\frac{3}{p}+1}_{p,1})}}^{h;|\Omega|\varepsilon}\\
&\qquad\qquad\qquad\qquad+\|m\|_{{L^2(0,T;\widehat{\dot{B}}{}^{\frac{3}{p}-1}_{p,1})}\cap {L^2(0,T;\widehat{\dot{B}}{}^{\frac{3}{p}}_{p,1})}\cap {L^1(0,T;\widehat{\dot{B}}{}^{\frac{3}{p}+1}_{p,1})}}^{h;|\Omega|\varepsilon}.
\end{align}
Now, we first choose $C_0$ sufficiently large to make $\n{\varepsilon a}_{L^{\infty}(0,T;L^{\infty})} \leq 1/2$ hold for $(a,m)\in Z^\varepsilon_T$ due to the embedding $\widehat{\dot{B}}{}^{\frac{3}{p}}_{p,1}(\mathbb{R}^3)\hookrightarrow L^\infty(\mathbb{R}^3)$, and to enable us to use \lemref{lemm-comp} to bound the composite functions in the subsequent analysis.
Using \lemref{keyproduct1} and \lemref{lemm-comp}, we have\begin{align}
	\|\div(m\otimes m)\|_{{L^1(0,T;\widehat{\dot{B}}{}^{\frac{3}{p}-3}_{p,1})}} &\leq C \|m\otimes m\|_{{L^1(0,T;\widehat{\dot{B}}{}^{\frac{3}{p}-2}_{p,1})}}
    \\&\leq C\|m\|_{{L^2(0,T;\widehat{\dot{B}}{}^{\frac{3}{p}-1}_{p,1})}}^2,\\
	\|\div(m\otimes m)\|_{{L^1(0,T;\widehat{\dot{B}}{}^{\frac{3}{p}-1}_{p,1})}}&\leq C \|m\otimes m\|_{{L^1(0,T;\widehat{\dot{B}}{}^{\frac{3}{p}}_{p,1})}}
    \\&\leq C\|m\|_{{L^2(0,T;\widehat{\dot{B}}{}^{\frac{3}{p}}_{p,1})}}^2,\\
		\|\div(I(\varepsilon a)m\otimes m)\|_{{L^1(0,T;\widehat{\dot{B}}{}^{\frac{3}{p}-3}_{p,1})}}&\leq C \|I(\varepsilon a)m\otimes m\|_{{L^1(0,T;\widehat{\dot{B}}{}^{\frac{3}{p}-2}_{p,1})}}\\
        &\leq C\|I(\varepsilon a)\|_{{L^\infty(0,T;\widehat{\dot{B}}{}^{\frac{3}{p}}_{p,1})}}\|m\otimes m\|_{{L^1(0,T;\widehat{\dot{B}}{}^{\frac{3}{p}-2}_{p,1})}}
        \\&\leq C\|\varepsilon\nabla a\|_{{L^\infty(0,T;\widehat{\dot{B}}{}^{\frac{3}{p}-1}_{p,1})}}\|m\|_{{L^2(0,T;\widehat{\dot{B}}{}^{\frac{3}{p}-1}_{p,1})}}^2,\\
	\|\div(I(\varepsilon a)m\otimes m)\|_{{L^1(0,T;\widehat{\dot{B}}{}^{\frac{3}{p}-1}_{p,1})}}&\leq C\|I(\varepsilon a)m\otimes m\|_{{L^1(0,T;\widehat{\dot{B}}{}^{\frac{3}{p}}_{p,1})}}   \\&\leq C\|I(\varepsilon a)\|_{{L^\infty(0,T;\widehat{\dot{B}}{}^{\frac{3}{p}}_{p,1})}}\|m\otimes m\|_{{L^1(0,T;\widehat{\dot{B}}{}^{\frac{3}{p}}_{p,1})}}
    \\&\leq C\|\varepsilon\nabla a\|_{{L^\infty(0,T;\widehat{\dot{B}}{}^{\frac{3}{p}-1}_{p,1})}}\|m\|_{{L^2(0,T;\widehat{\dot{B}}{}^{\frac{3}{p}}_{p,1})}}^2,\\
	\|\Delta(I(\varepsilon a)m)\|_{{L^1(0,T;\widehat{\dot{B}}{}^{\frac{3}{p}-3}_{p,1})}}&+\|\nabla\div(I(\varepsilon a)m)\|_{{L^1(0,T;\widehat{\dot{B}}{}^{\frac{3}{p}-3}_{p,1})}}
    \\ &\leq C\|I(\varepsilon a)m\|_{{L^1(0,T;\widehat{\dot{B}}{}^{\frac{3}{p}-1}_{p,1})}}
    \\
	&\leq C\|\varepsilon\nabla a\|_{{L^2(0,T;\widehat{\dot{B}}{}^{\frac{3}{p}-1}_{p,1})}}\|m\|_{{L^2(0,T;\widehat{\dot{B}}{}^{\frac{3}{p}-1}_{p,1})}},\\
	\|\Delta(I(\varepsilon a)m)\|_{{L^1(0,T;\widehat{\dot{B}}{}^{\frac{3}{p}-1}_{p,1})}}&+\|\nabla\div(I(\varepsilon a)m)\|_{{L^1(0,T;\widehat{\dot{B}}{}^{\frac{3}{p}-1}_{p,1})}}
    \\&\leq C\|I(\varepsilon a)m\|_{{L^1(0,T;\widehat{\dot{B}}{}^{\frac{3}{p}+1}_{p,1})}}
    \\
	&\leq C\|\varepsilon\nabla a\|_{{L^\infty(0,T;\widehat{\dot{B}}{}^{\frac{3}{p}-1}_{p,1})}}\|m\|_{{L^1(0,T;\widehat{\dot{B}}{}^{\frac{3}{p}+1}_{p,1})}}\\&\qquad+\|m\|_{{L^2(0,T;\widehat{\dot{B}}{}^{\frac{3}{p}}_{p,1})}}\|\varepsilon\nabla a\|_{{L^2(0,T;\widehat{\dot{B}}{}^{\frac{3}{p}}_{p,1})}},\\
	\frac{1}{\varepsilon}\|J(\varepsilon a)\nabla a\|_{{L^1(0,T;\widehat{\dot{B}}{}^{\frac{3}{p}-3}_{p,1})}}
    &\leq \frac{C}{\varepsilon^2}\|G(\varepsilon a)\|_{{L^1(0,T;\widehat{\dot{B}}{}^{\frac{3}{p}-2}_{p,1})}}\\
    &\leq C\|a^2\|_{{L^1(0,T;\widehat{\dot{B}}{}^{\frac{3}{p}-2}_{p,1})}}+C\|H(\varepsilon a)a^2\|_{{L^1(0,T;\widehat{\dot{B}}{}^{\frac{3}{p}-2}_{p,1})}}\\
    &\leq C\left(1+\|H(\varepsilon a)\|_{{L^\infty(0,T;\widehat{\dot{B}}{}^{\frac{3}{p}}_{p,1})}}\right)\|a^2\|_{{L^1(0,T;\widehat{\dot{B}}{}^{\frac{3}{p}-2}_{p,1})}}
    \\&\leq C\left(1+\|\varepsilon\nabla a\|_{{L^\infty(0,T;\widehat{\dot{B}}{}^{\frac{3}{p}-1}_{p,1})}}\right)\|a\|_{{L^{2}(0,T;\widehat{\dot{B}}{}^{\frac{3}{p}-1}_{p,1})}}^2,\\
   &\text{(here we have used \eqref{thegfunction} and \eqref{thegfunction1})}\\
	\frac{1}{\varepsilon}\|J(\varepsilon a)\nabla a\|_{{L^1(0,T;\widehat{\dot{B}}{}^{\frac{3}{p}-1}_{p,1})}}
    &\leq \frac{C}{\varepsilon}\|J(\varepsilon a)\|_{{L^2(0,T;\widehat{\dot{B}}{}^{\frac{3}{p}}_{p,1})}}\|\nabla a\|_{{L^2(0,T;\widehat{\dot{B}}{}^{\frac{3}{p}-1}_{p,1})}}
    \\&\leq\|a\|_{{L^2(0,T;\widehat{\dot{B}}{}^{\frac{3}{p}}_{p,1})}}^2,
    \\
	\varepsilon^2\|\nabla(a\Delta a)\|_{{L^1(0,T;\widehat{\dot{B}}{}^{\frac{3}{p}-3}_{p,1})}}&\leq C\varepsilon^2\|a\Delta a\|_{{L^1(0,T;\widehat{\dot{B}}{}^{\frac{3}{p}-2}_{p,1})}}\\
	&\leq C\|\varepsilon\nabla a\|_{{L^2(0,T;\widehat{\dot{B}}{}^{\frac{3}{p}-1}_{p,1})}}^2,\\
	\varepsilon^2\|\nabla(a\Delta a)\|_{{L^1(0,T;\widehat{\dot{B}}{}^{\frac{3}{p}-1}_{p,1})}}&\leq C\varepsilon^2\|a\Delta a\|_{{L^1(0,T;\widehat{\dot{B}}{}^{\frac{3}{p}}_{p,1})}}\\
	&\leq C\|\varepsilon\nabla a\|_{{L^\infty(0,T;\widehat{\dot{B}}{}^{\frac{3}{p}-1}_{p,1})}}\|\varepsilon\nabla a\|_{{L^1(0,T;\widehat{\dot{B}}{}^{\frac{3}{p}+1}_{p,1})}}\\
	\varepsilon^2\|\nabla(|\nabla a|^2)\|_{{L^1(0,T;\widehat{\dot{B}}{}^{\frac{3}{p}-3}_{p,1})}}
	&+\varepsilon^2\|\div(\nabla a\otimes\nabla a)\|_{{L^1(0,T;\widehat{\dot{B}}{}^{\frac{3}{p}-3}_{p,1})}}\\
&\leq C\varepsilon^2\|\nabla a\otimes\nabla a\|_{{L^1(0,T;\widehat{\dot{B}}{}^{\frac{3}{p}-2}_{p,1})}}
    \\
	&\leq C\|\varepsilon\nabla a\|_{{L^2(0,T;\widehat{\dot{B}}{}^{\frac{3}{p}-1}_{p,1})}}^2,\\
	\varepsilon^2\|\nabla(|\nabla a|^2)\|_{{L^1(0,T;\widehat{\dot{B}}{}^{\frac{3}{p}-1}_{p,1})}}
	&+\varepsilon^2\|\div(\nabla a\otimes\nabla a)\|_{{L^1(0,T;\widehat{\dot{B}}{}^{\frac{3}{p}-1}_{p,1})}}\\
&\leq\varepsilon^2\|\nabla a\otimes\nabla a\|_{{L^1(0,T;\widehat{\dot{B}}{}^{\frac{3}{p}}_{p,1})}}
    \\
	&\leq C\|\varepsilon\nabla a\|_{{L^\infty(0,T;\widehat{\dot{B}}{}^{\frac{3}{p}-1}_{p,1})}}\|\varepsilon\nabla a\|_{{L^1(0,T;\widehat{\dot{B}}{}^{\frac{3}{p}+1}_{p,1})}}.
\end{align}
Then, thanks to \lemref{energylemma}, the fact $|\Omega|\varepsilon\leq 1$ and the nonlinear estimates above, we get\begin{align}
\|\Phi_{\varepsilon}[a,m]\|_{Z^\varepsilon_T}
&\leq \|(\widetilde{a_\varepsilon},\widetilde{m_\varepsilon})\|_{Z^\varepsilon_T}+\left\|\int_{0}^{t}\mathcal{G}_{\varepsilon}(t-\tau)\binom{0}{N_\varepsilon[a,m]}(\tau)d\tau\right\|_{Z^\varepsilon_T}^{\ell;|\Omega|\varepsilon}\\
&\qquad+\left\|\int_{0}^{t}\mathcal{G}_{\varepsilon}(t-\tau)\binom{0}{N_\varepsilon[a,m]}(\tau)d\tau\right\|_{Z^\varepsilon_T}^{h;|\Omega|\varepsilon}
\\&\leq \|(\widetilde{a_\varepsilon},\widetilde{m_\varepsilon})\|_{Z^\varepsilon_T}+C\left\|N_\varepsilon[a,m]\right\|_{{L^1(0,T;\widehat{\dot{B}}{}^{\frac{3}{p}-3}_{p,1})}\cap {L^1(0,T;\widehat{\dot{B}}{}^{\frac{3}{p}-2}_{p,1})}\cap{L^1(0,T;\widehat{\dot{B}}{}^{\frac{3}{p}-1}_{p,1})}}^{\ell;|\Omega|\varepsilon}\\
&\qquad+C\left\|N_\varepsilon[a,m]\right\|_{{L^1(0,T;\widehat{\dot{B}}{}^{\frac{3}{p}-3}_{p,1})}\cap {L^1(0,T;\widehat{\dot{B}}{}^{\frac{3}{p}-2}_{p,1})}\cap{L^1(0,T;\widehat{\dot{B}}{}^{\frac{3}{p}-1}_{p,1})}}^{h;|\Omega|\varepsilon}\\&\leq \|(\widetilde{a_\varepsilon},\widetilde{m_\varepsilon})\|_{Z^\varepsilon_T}+C\left\|N_\varepsilon[a,m]\right\|_{{L^1(0,T;\widehat{\dot{B}}{}^{\frac{3}{p}-3}_{p,1})}{L^1(0,T;\widehat{\dot{B}}{}^{\frac{3}{p}-1}_{p,1})}}\\
&\leq \|(\widetilde{a_\varepsilon},\widetilde{m_\varepsilon})\|_{Z^\varepsilon_T}+C_0\|(a,m)\|_{Z^\varepsilon_T}^2
\end{align}
by taking $C_0$ sufficiently large, where we have set
\begin{align}
\binom{\widetilde{a_\varepsilon}(t)}{\widetilde{m_\varepsilon}(t)}:=\mathcal{G}_{\varepsilon}(t)\binom{a_0}{m_0}.
\end{align}
Using Berstein inequalities and \lemref{energylemma}, and choosing a sufficiently large constant $\alpha\gg1$ such that
\begin{align}
	&
    \|\varepsilon\nabla \widetilde{a_\varepsilon}\|_{{\widetilde{L^\infty}(0,\infty;\widehat{\dot{B}}{}^{\frac{3}{p}-1}_{p,1})}}
	\\
    &\quad 
    =\|\varepsilon\nabla \widetilde{a_\varepsilon}\|_{{\widetilde{L^\infty}(0,\infty;\widehat{\dot{B}}{}^{\frac{3}{p}-1}_{p,1})}}^{\ell;|\Omega|\varepsilon}+\|\varepsilon\nabla \widetilde{a_\varepsilon}\|_{{\widetilde{L^\infty}(0,\infty;\widehat{\dot{B}}{}^{\frac{3}{p}-1}_{p,1})}}^{m;|\Omega|\varepsilon,\alpha}+\|\varepsilon\nabla \widetilde{a_\varepsilon}\|_{{\widetilde{L^\infty}(0,\infty;\widehat{\dot{B}}{}^{\frac{3}{p}-1}_{p,1})}}^{h;\alpha}
	\\
    &\quad 
    \leq \varepsilon\alpha\left(\| \widetilde{a_\varepsilon}\|_{{\widetilde{L^\infty}(0,\infty;\widehat{\dot{B}}{}^{\frac{3}{p}-1}_{p,1})}}^{\ell;|\Omega|\varepsilon}+\| \widetilde{a_\varepsilon}\|_{{\widetilde{L^\infty}(0,\infty;\widehat{\dot{B}}{}^{\frac{3}{p}-1}_{p,1})}}^{m;|\Omega|\varepsilon,\alpha}\right)+\|\varepsilon\nabla \widetilde{a_\varepsilon}\|_{{\widetilde{L^\infty}(0,\infty;\widehat{\dot{B}}{}^{\frac{3}{p}-1}_{p,1})}}^{h;\alpha}
	\\
    &\quad 
    \leq C\varepsilon\alpha\|(a_0,\varepsilon \nabla a_0,m_0)\|_{\widehat{\dot{B}}{}^{\frac{3}{p}-1}_{p,1}}+C\|(a_0,\varepsilon \nabla a_0,m_0)\|_{\widehat{\dot{B}}{}^{\frac{3}{p}-1}_{p,1}}^{h;\alpha}
    \\
    &\quad 
    \leq C\varepsilon\alpha\|(a_0,\varepsilon \nabla a_0,m_0)\|_{\widehat{\dot{B}}{}^{\frac{3}{p}-1}_{p,1}}+\frac{1}{16C_0}.
\end{align}
Hence, there exists a positive constant $\varepsilon_0$ such that\begin{equation}
	\begin{split}
		\|\varepsilon\nabla \widetilde{a_\varepsilon}\|_{{\widetilde{L^\infty}(0,\infty;\widehat{\dot{B}}{}^{\frac{3}{p}-1}_{p,1})}}\leq \frac{1}{8C_0}
	\end{split}
\end{equation}for all $0<\varepsilon\leq\varepsilon_0$. Such $\varepsilon$ is fixed from now on. Reverting to \eqref{hkjsll} and noticing that $|\Omega|\varepsilon\leq 1$, we see that\begin{align}
&\|\dot{\Delta}_j(\widetilde{a_\varepsilon},\varepsilon\nabla \widetilde{a_\varepsilon},\widetilde{m_\varepsilon})\|_{\widehat{L^p}}\\
&\leq Ce^{-\frac{2^{4j}}{C(\Omega^2\varepsilon^2+2^{2j})}t}\|\dot{\Delta}_j(a_0,\varepsilon\nabla a_0,m_0)\|_{\widehat{L^p}}\\
&\leq \begin{cases}
	Ce^{-C^{-1}2^{4j}t}\|\dot{\Delta}_j(a_0,\varepsilon\nabla a_0,m_0)\|_{\widehat{L^p}}\quad&\text{if}\quad 2^j\leq 1,  \\
		Ce^{-C^{-1}2^{2j}t}\|\dot{\Delta}_j(a_0,\varepsilon\nabla a_0,m_0)\|_{\widehat{L^p}}\quad&\text{if}\quad 2^j> 1.
\end{cases}
\end{align}
Then, by virtue of Lebesgue dominated convergence theorem, there exists a positive time $T_0$ independent of $\Omega$ and $\varepsilon$ such that
\begin{align}
	&\| \widetilde{a_\varepsilon}\|_{{L^2(0,T_0;\widehat{\dot{B}}{}^{\frac{3}{p}-1}_{p,1})}\cap{L^{2}(0,T_0;\widehat{\dot{B}}{}^{\frac{3}{p}}_{p,1})}}+\|\varepsilon\nabla \widetilde{a_\varepsilon}\|_{{L^2(0,T_0;\widehat{\dot{B}}{}^{\frac{3}{p}-1}_{p,1})}\cap {L^2(0,T_0;\widehat{\dot{B}}{}^{\frac{3}{p}}_{p,1})}\cap {L^1(0,T_0;\widehat{\dot{B}}{}^{\frac{3}{p}+1}_{p,1})}}\\
	&\qquad+\|\widetilde{m_\varepsilon}\|_{{L^2(0,T_0;\widehat{\dot{B}}{}^{\frac{3}{p}-1}_{p,1})}\cap {L^2(0,T_0;\widehat{\dot{B}}{}^{\frac{3}{p}}_{p,1})}\cap {L^1(0,T_0;\widehat{\dot{B}}{}^{\frac{3}{p}+1}_{p,1})}}\leq \frac{1}{8C_0}.
\end{align}
Hence, we have\begin{align}
\|(\widetilde{a_\varepsilon},\widetilde{m_\varepsilon})\|_{Z^\varepsilon_{T_0}}\leq \frac{1}{4C_0}.
\end{align}
At this stage, we arrive at\begin{align}
\|\Phi_{\varepsilon}[a,m]\|_{Z^\varepsilon_{T_0}}&\leq\|(\widetilde{a_\varepsilon},\widetilde{m_\varepsilon})\|_{Z^\varepsilon_{T_0}}+C_0\|(a,m)\|_{Z^\varepsilon_{T_0}}^2\\
&\leq \frac{1}{4C_0}+C_0\left(\frac{1}{2C_0}\right)^2=\frac{1}{2C_0},
\end{align}which indicates that $\Phi_{\varepsilon}[a,m]\in Z^\varepsilon_{T_0}$ for $(a,m)\in Z^\varepsilon_{T_0}$. In a similar fashion, we also can conclude that there exists a positive constant $C_3$ such that
\begin{align}
\|&\Phi_{\varepsilon}[a_2,m_2]-\Phi_{\varepsilon}[a_1,m_1]\|_{Z^\varepsilon_{T_0}}\\&\leq C_3\left(\|(a_2,m_2)\|_{Z^\varepsilon_{T_0}}+\|(a_1,m_1)\|_{Z^\varepsilon_{T_0}}\right)\|(a_2-a_1,m_2-m_1)\|_{Z^\varepsilon_{T_0}}
\end{align}for all $(a_2,m_2)$, $(a_1,m_1)\in Z^\varepsilon_{T_0}$. Then, taking $C_0$ large enough, it holds that\begin{align}
\|&\Phi_{\varepsilon}[a_2,m_2]-\Phi_{\varepsilon}[a_1,m_1]\|_{Z^\varepsilon_{T_0}}\\
&\leq\frac{C_0}{2}\left(\frac{1}{2C_0}+\frac{1}{2C_0}\right)\|(a_2-a_1,m_2-m_1)\|_{Z^\varepsilon_{T_0}}\\
&=\frac{1}{2}\|(a_2-a_1,m_2-m_1)\|_{Z^\varepsilon_{T_0}}.
\end{align}
Thus, the Banach fixed-point theorem indicates that there exists a unique $(a,m)\in Z^\varepsilon_{T_0}$ such that $(a,m)=\Phi_{\varepsilon}[a,m]$, which completes the proof.
\end{proof}

\ \\

\noindent
{\bf Data availability.} \\
Data sharing not applicable to this article as no datasets were generated or analysed during the current study.

\noindent
{\bf Conflict of interest.} \\
The authors have declared no conflicts of interest.

\noindent
{\bf Acknowledgements.} \\
The first author was supported by Grant-in-Aid for Research Activity Start-up, Grant Number JP23K19011. The second author was supported by Wuyi University Research Foundation, Grant Number BSQD2226.

\bibliographystyle{abbrv} 
\bibliography{ref}

\end{document}